\documentclass[11pt]{amsart}
\usepackage{amssymb,amsmath,amsthm,graphics,amscd,amsfonts}
\usepackage{color}

\theoremstyle{plain}
\newtheorem{theorem}{Theorem}[section]
\newtheorem{proposition}[theorem]{Proposition}
\newtheorem{corollary}[theorem]{Corollary}

\newtheorem{claim}[theorem]{Claim}

\newtheorem{definition}[theorem]{Definition}

\newtheorem{remark}[theorem]{Remark}

\newcommand{\bfC}{{\mathbb C}}
\newcommand{\bfP}{{\mathbb P}}
\newcommand{\bfR}{{\mathbb R}}
\newcommand{\bfZ}{{\mathbb Z}}

\newcommand{\bfQ}{{\mathbb Q}}

\newcommand{\barj}{{\overline j}}

\newcommand{\barl}{{\overline \ell}}

\newcommand{\barpartial}{{\overline \partial}}

\newcommand{\mapright}[1]{\smash{\mathop{   \hbox to 0.7cm{\rightarrowfill}}
  \limits^{#1}}}

\newcommand{\Ric}{\operatorname{Ric}}

\def\C{\mathbb C}

\newcommand{\diam}{\operatorname{diam}}
\newcommand{\grad}{\operatorname{grad}}

\renewcommand{\emph}[1]{{\color{red} \it #1}}

\newcommand{\Q}{\mathbb{Q}}

\renewcommand{\epsilon}{\varepsilon}

\renewcommand{\P}{\mathbb{P}}

\title{Fano-Ricci limit spaces and spectral convergence}
\author{Akito Futaki, Shouhei Honda and Shunsuke Saito}
\address{Graduate School of Mathematical Sciences, The University of Tokyo, 3-8-1 Komaba Meguro-ku Tokyo 153-8914, Japan}
\email{afutaki@ms.u-tokyo.ac.jp}
\address{Mathematical Institute, Tohoku University, Aoba-ku, Sendai, 980-8578, Japan}
\email{shonda@m.tohoku.ac.jp}
\address{Graduate School of Mathematical Sciences, The University of Tokyo, 3-8-1 Komaba Meguro-ku Tokyo 153-8914, Japan}
\email{ssaito@ms.u-tokyo.ac.jp}

\begin{document}
\begin{abstract}
We study the behavior under Gromov-Hausdorff convergence 
of the spectrum of weighted $\barpartial$-Laplacian on compact K\"ahler manifolds.
This situation typically occurs for a sequence of Fano manifolds with anticanonical
K\"ahler class. We apply it to show that, if an almost smooth Fano-Ricci limit space
admits a K\"ahler-Ricci limit soliton and
the space of all $L^2$ holomorphic vector fields with smooth potentials is a Lie algebra with respect to the Lie bracket, 
then the Lie algebra has the same structure as smooth K\"ahler-Ricci
solitons. In particular if a $\Q$-Fano variety admits a K\"ahler-Ricci limit soliton and all holomorphic 
vector fields are $L^2$ with smooth potentials then the Lie algebra has the same structure as smooth K\"ahler-Ricci
solitons. If the sequence consists of K\"ahler-Ricci solitons then the Ricci limit space 
is a weak K\"ahler-Ricci
soliton on a $\mathbb{Q}$-Fano variety and the space of 
limits of $1$-eigenfunctions for the weighted $\barpartial$-Laplacian
forms a Lie algebra with respect to the Poisson bracket and admits a similar decomposition 
as smooth K\"ahler-Ricci solitons.
\end{abstract}
\maketitle
\section{Introduction}
In this paper we study the behavior of the spectrum of a weighted $\barpartial$-Laplacian 
under Gromov-Hausdorff convergence of a sequence of compact K\"ahler manifolds. 
Typically we consider a sequence $(X_i, g_i)$ of Fano manifolds $X_i$ and K\"ahler metrics $g_i$
where the K\"ahler form $\omega_i$ of $g_i$ represents $2\pi c_1(X_i)$, i.e. $2\pi$ times the
anti-canonical class. 
Since the Ricci form $\mathrm{Ric}(\omega_i)$ also represents $2\pi c_1(X_i)$
there is a real valued smooth function $F_i$, called the Ricci potential, 
given by
$$ \mathrm{Ric}(\omega_i) - \omega_i = i\partial\barpartial F_i.$$
The weighted $\barpartial$-Laplacian $\Delta^{F_i}_\barpartial$ we consider is given for a smooth 
function $u$ by
$$ \Delta^{F_i}_\barpartial u = e^{-F_i}\barpartial^\ast(e^{F_i}\barpartial u).$$
This is a self-adjoint elliptic operator with respect to the weighted measure $e^{F_i}dH$.
Assuming the one side bound of the Ricci curvature 
$$  \mathrm{Ric}(g_i) \ge Kg_i $$
for a constant $K$, upper diameter bound, uniform bound of  $||F_i||_{L^{\infty}}$ and 
$L^2$-strong convergence of $F_{i}$ and complex structures $J_i$ (see subsection 2.1 for
more detail),
we consider non-collapsing K\"ahler-Ricci limit space. When the sequence consists of
Fano manifolds with anti-canonical K\"ahler class the limit is called a {\it Fano-Ricci limit space}.

In the Riemannian 
case the behavior of the spectrum of the Laplacian 
under Gromov-Hausdorff convergence was studied by Fukaya \cite{Fukaya} and Cheeger and Colding 
\cite{CheegerColding3}. 
We see in section 3.1 that the spectral behavior for the weighted $\barpartial$-Laplacian is
continuous with respect to the Gromov-Hausdorff topology (Proposition 3.13).

On the Fano manifold $X_i$ we have the following Weitzenb\"ock formula:
\begin{align}
\int_{X_i}|\Delta^{F_i}_{\overline{\partial}}f_i|^2\,dH^n_{F_i} = \int_{X_i}|\nabla '' \mathrm{grad}'f_i|^2\,dH^n_{F_i}+\int_{X_i}|\overline{\partial}f_i|^2\,dH^n_{F_i}.
\end{align}
In Theorem \ref{weith} we show the following Weitzenb\"ock inequality on the Fano-Ricci limit
space:
\begin{align}\label{weitzon}
\int_X|\Delta^F_{\overline{\partial}}f|^2\,dH^n_F \ge \int_X|\nabla '' \mathrm{grad}'f|^2\,dH^n_F+\int_X|\overline{\partial}f|^2\,dH^n_F.
\end{align}
This implies the first non-zero eigenvalue $\lambda_1(\Delta_{\overline{\partial}}^F, X)$ of the weighted 
$\barpartial$-Laplacian $\Delta_{\overline{\partial}}^F$ on the limit space $X$ satisfies
$$\lambda_1(\Delta_{\overline{\partial}}^F, X) \ge 1,$$
and if $f$ is in the domain $\mathcal{D}^2_{\bfC}(\Delta_{\overline{\partial}}^F, X)$ of 
$\Delta_{\overline{\partial}}^F$ 
with $\Delta^F_{\overline{\partial}}f=f$, then
$\nabla''\mathrm{grad}'f=0$. In particular if $U$ is an open subset of $X$ and 
$(U, g_X|_U, J|_U)$ is a smooth K\"ahler manifold with $F|_U \in C^{\infty}(U)$, then $f|_U \in C^{\infty}_{\bfC}(U)$ and $\mathrm{grad}'f$ is a holomorphic vector field on $U$ (Corollary \ref{1steigen}).

For a smooth Fano manifold $M$, 
the Lie algebra $\mathfrak{hol}(M)$ of all holomorphic vector fields on $M$ 
is isomorphic to the space $\Lambda_1$ of the eigenfunctions corresponding
to the eigenvalue $1$ ($1$-eigenfunctions for short) 
for $\Delta_{\overline{\partial}}^F$ with the Poisson structure, see \cite{Futaki87, Futaki88, Futaki15}.
The Poisson structure can be defined
on the noncollapsed K\"ahler-Ricci limit space $X$ as in Definition \ref{poissondefinition}. 
However there is a difficulty in finding if the space of all $\mathrm{grad}'f$ obtained as above forms a 
Lie algebra since
it is not clear if the space of $1$-eigenfunctions of $\Delta_{\overline{\partial}}^F$ on the limit space
is closed under Poisson bracket. A key to overcome this difficulty is to see when the Lie bracket of
two $L^2$ vector fields of the form $\mathrm{grad}'f$ as obtained above becomes $L^2$ again. 

If a smooth Fano manifold $M$ admits a K\"ahler-Ricci soliton $\omega$, i.e. 
$\mathrm{Ric}(\omega) - \omega = i \partial\barpartial F$ with $\mathrm{grad}'F$ is a 
holomorphic vector field, then the Lie algebra $\mathfrak{hol}(M)$ 
of all holomorphic vector fields on $M$ is known to have the following structure (\cite{TianZhu00}):
\[\mathfrak{hol}(M)=\mathfrak{hol}_0(M) \oplus \bigoplus_{\alpha >0}\mathfrak{hol}_{\alpha}(M),\]
where $\mathfrak{hol}_{\alpha}(M)$ is the $\alpha$-eigenspace of the adjoint action of $-\mathrm{grad}'F$.
Furthermore,  $\mathfrak{hol}_0(M)$ is a maximal reductive Lie subalgebra. Note that the direct sum decomposition is meant 
as a vector space and 
$[\mathfrak{hol}_{\alpha}(M), \mathfrak{hol}_{\beta}(M)] \subset 
\mathfrak{hol}_{\alpha + \beta}(M)$
holds as a Lie algebra.

If a sequence of K\"ahler-Ricci solitons $X_i$ converges to $X$ then the limit $X$ 
is a K\"ahler-Ricci limit soliton
which is also a $\mathbb{Q}$-Fano variety. Then the vector space $\Lambda$ consisting 
of the limit of the 1-eigenfunctions (eigenfunctions with eigenvalue $1$) on $X_i$ form a Lie algebra, 
and a similar decomposition theorem holds for $\Lambda$ on 
the limit $X$ as in the case of smooth K\"ahler-Ricci solitons as described above (Theorem \ref{decompconti}). This is proved by showing that the 
$L^2$-Lie bracket
property as mentioned above is satisfied.

If the Ricci limit space $X$ is a $\mathbb{Q}$-Fano variety then holomorphic vector fields
on the regular part extend
to $X$ and they form a Lie algebra $\mathfrak{hol}(X)$. 
If we assume they are all $L^2$ with smooth potentials, then they are gradient vector
fields of  $1$-eigenfunctions (Proposition \ref{qfano}). Thus, assuming that the limit 
$X$ is 
a K\"ahler-Ricci soliton on the regular part and that the Lie algebra $\mathfrak{hol}(X)$ 
consists of $L^2$ holomorphic vector fields with smooth potentials on the regular part,  $\mathfrak{hol}(X)$ has a
similar decomposition as in the smooth K\"ahler-Ricci solitons (Theorem \ref{mainth}).

The theory of 
Cheeger-Colding \cite{CheegerColding1, CheegerColding2, CheegerColding3} and Cheeger-Colding-Tian \cite{CheegerColdingTian} has been applied to
complex geometry by Donaldson-Sun \cite{DonaldsonSun2012, DonaldsonSun2015}
where the two side bound of Ricci curvature 
$$ - C g \le \mathrm{Ric} \le Cg $$
for some constant $C > 0$ is assumed. Further the theory of Cheeger, Colding and Tian 
was used to prove Yau-Tian-Donaldson conjecture by
Chen-Donaldson-Sun \cite{CDS1, CDS2, CDS3}, Tian \cite{Tian12},
and to study the compactification of the moduli space of K\"ahler-Einstein manifolds
in \cite{OSS}, \cite{SSY14}, \cite{Odaka14}, \cite{LiWangXu2015}.
The difference between these works and ours is that we employ one side lower bound of Ricci curvature,
and the two side bound ensures a $C^{2,\alpha}$ differentiable structure on an open set 
of the limit
while the one side lower bound only ensures a weak $C^{1, 1}$ 
structure outside the singular set of measure
zero.
In particular, except for section 5 and section 6, we do not require the openness of the regular set.

This paper is organized as follows.

In Section 2, we consider $L^p$-convergence for $\bfC$-valued tensor fields in the Gromov-Hausdorff setting, which is known on the real setting in \cite{Honda13}.
Main results are Rellich compactness (Theorem \ref{Rellich}) and Proposition \ref{L2norm}, which play key roles to prove the spectral convergence of the weighted $\barpartial$-Laplacian.

In Section 3, we first establish the spectral convergence of the weighted $\barpartial$-Laplacian on general setting. 
Second, we define the covariant derivative $\nabla''$ on a noncollapsed K\"ahler-Ricci limit space, which is a key notion to establish the Weitzenb\"ock inequality (\ref{weitzon}) on a Fano-Ricci limit space. The essential idea of the definition of $\nabla''$ is based on Gigli's approach to nonsmooth differential geometry discussed in \cite{Gigli} via the regularity theory of the heat flow on $RCD$-spaces by 
Ambrosio-Gigli-Savar\'e \cite{AGS}.
A main property of $\nabla''$ is the $L^2$-weak stability (Proposition \ref{stability}), which plays a key role in the proof of the Weitzenb\"ock inequality (\ref{weitzon}). 

In Section 4, we first prove the Weitzenb\"ock inequality (\ref{weitzon}). Next we discuss the regularity of the Ricci potential on a Fano manifold with a lower Ricci curvature bound.
As a corollary, we establish a compactness with respect to the Gromov-Hausdorff topology, which states 
that a sequence of Fano manifolds with uniform lower bound of the Ricci curvature, uniform lower 
bound of the volumes, uniform lower bound of the Ricci potentials, and uniform upper bound of 
the diameters, has a convergent subsequence to a Fano-Ricci limit space (Corollary \ref{compactness}).

We also discuss in Section 4 the behavior of holomorphic vector fields and the Futaki invariant with respect to the Gromov-Hausdorff topology. As a corollary, we give a new uniform bound of the dimension of the space of all holomorphic vector fields on a Fano manifold (Corollary \ref{dimensionbound}).
The final subsection of Section 4 is devoted to constructing a Lie algebra consisting of $L^2$-holomorphic vector fields with smooth potentials on a nonsmooth Fano-Ricci limit space. 
As we mentioned above, this has a difficulty in showing that 
the Lie bracket of two $L^2$-vector fields is $L^2$.
Proposition \ref{poissonstr} and Corollary \ref{liestr} are used to overcome this difficulty by 
considering the limit holomorphic vector fields.

In Section 5, under an additional assumption that a Fano-Ricci limit space is almost smooth, we study holomorphic vector fields in further detail.
In particular, we establish a similar decomposition as smooth K\"ahler-Ricci solitons 
as in \cite{TianZhu00} (Theorem \ref{decomp}).
It is worth pointing out that $\mathfrak{hol}_0(X)$ being reductive Lie subalgebra in the decomposition theorem comes from a Cheeger-Colding's result in  \cite{CheegerColding2} that the isometry group of a noncollapsed 
Ricci limit space is a Lie group.

In the final section, Section 6, we consider the case that a Fano-Ricci limit space is a $\bfQ$-Fano variety.
Then we prove that the space is an almost smooth in the sense of Section 5 (Proposition \ref{qfano}).
Thus we can apply the decomposition theorem in Section 5 to this case (Theorem \ref{mainth}).
We also check that combining results above with Phong-Song-Sturm's recent 
work \cite{PhongSongSturm}, this situation typically occurs if a sequence we consider 
consists of K\"ahler-Ricci solitons (Theorem \ref{decompconti}).

\section{Noncollapsed weighted K\"ahler-Ricci limit spaces}
In this section we discuss the spectral behavior of K\"ahler manifolds with respect 
to the Gromov-Hausdorff topology. Recall that 
a sequence of compact metric spaces $(X_i, d_{X_i})$ is said to Gromov-Hausdorff converges 
to a compact metric space $(X, d_X)$ if there exist a sequence of positive numbers $\epsilon_i$ 
with $\epsilon_i \to 0$ and a sequence of maps $\phi_i:X_i \to X$ 
such that 
\begin{enumerate}
\item[(i)]
$X=B_{\epsilon_i}(\phi_i(X_i))$ where $B_r(A)$ is the $r$-neighborhood of $A$, and 
\item[(ii)]
$|d_{X_i}(x, y)-d_X(\phi_i(x), \phi_i(y))|<\epsilon_i$ holds for any $i$ 
and $x, y \in X_i$ where $d_X$ is the distance function of $X$.
\end{enumerate}
Then for a sequence $x_i \in X_i$ and a point $x \in X$ 
we denote $x_i \stackrel{GH}{\to} x$ if $\phi_i(x_i) \to x$ in $X$.

Moreover for a sequence of Borel regular measures $\upsilon_i$ on $X_i$ and a Borel regular measure $\upsilon$ on $X$, $(X_i, d_{X_i}, \upsilon_i)$ is said to converge to $(X, d_X, \upsilon)$ in the measured Gromov-Hausdorff sense if 
\[\lim_{i \to \infty}\upsilon_i(B_r(x_i))=\upsilon (B_r(x))\]
holds for any $r>0$ and $x_i \stackrel{GH}{\to} x$.
\subsection{Setting}
Our setting in this section is the following;
\begin{enumerate}
\item[(2.1a)] Let $K \in \bfR$ and let $d>0$.
\item[(2.1b)] Let $(X_i, g_{X_i}, J_i)$ be a sequence of $m$-dimensional compact K\"ahler manifolds with $\Ric_{X_i} \ge Kg_{X_i}$ and $\diam X_i \le d$ where $g_{X_i}$, $J_i$ and $\Ric_{X_i}$ respectively 
denote a Riemannian metric, a complex structure and
the Ricci curvature of $X_i$ with $(g_i, J_i)$ giving a K\"ahler structure. We put $n = 2m$.
\item[(2.1c)] Let $X$ be the Gromov-Hausdorff limit of $(X_i, g_{X_i})$ and let $g_X$ denotes the (canonical) Riemannian metric in a weak sense (we give an explanation below).
\item[(2.1d)] Let $F_i$ be a sequence of real valued functions $F_i \in L^{\infty}(X_i)$ with $L:=\sup_i||F_i||_{L^{\infty}}<\infty$.
\item[(2.1e)] Let $F, J$ be the $L^2$-strong limits of $F_{i}, J_i$ on $X$, respectively. See Definition \ref{Lpconv} and \ref{Lpreal} below for the meaning of strong convergence.
\end{enumerate}
In this setting it was shown in \cite[Theorem $6.19$]{Honda14b}  that $(X, g_X)$ is the noncollapsed limit of $(X_i, g_{X_i})$, i.e. the Hausdorff (or topological) dimension of $X$ is equal to $n$ and that $J \circ J=-id$ in $L^{\infty}(TX \otimes T^*X) \simeq L^{\infty}(\mathrm{End}TX)$.
In particular it follows from \cite[Theorem $5.9$]{CheegerColding1} that $(X_i, g_{X_i}, H^n)$ converges to $(X, g_{X}, H^n)$ in the measured Gromov-Hausdorff sense with $0<H^n(X)<\infty$ where $H^n$ denotes the 
$n$-dimensional Hausdorff measure.
Note that for every sequence $G_i \in C^0(\bfR)$ which converges uniformly to $G$ on every compact subset of $\bfR$, $G_i(F_i)$ $L^p$-converges strongly to $G(F)$ on $X$ for every $p \in (1, \infty)$ (c.f. \cite[Proposition $4.1$]{Honda11}).
In particular, $(X_i, g_{X_i}, e^{F_i}H^n)$ converges to $(X, g_X, e^F H^n)$ in the measured Gromov-Hausdorff sense.

From now on we give a short introduction of the study of Ricci limit spaces which are Gromov-Hausdorff limit spaces with Ricci curvature bounded below.

Cheeger and Colding proved that $(X, H^n)$ is rectifiable (see \cite{CheegerColding3}, section 5, (i), (ii), (iii)).
In particular 
we can construct the (real) tangent bundle
\[\pi: TX \to X\]
and define the canonical metric $g_X$ on each fibers.
Note that the fibers $T_xX$ are well-defined at a.e. $x \in X$ and that
$g_X$ is compatible with the metric structure in the following sense; 
For every Lipschitz function $f$ on an open subset $U$ of $X$ there exists a gradient vector field $\grad f(x)$ which is well-defined at a.e. $x \in U$ such that
\[|\grad f|(x):=\sqrt{g_X(\grad f, \grad f)(x)}=\lim_{y \to x}\left( \frac{|f(x)-f(y)|}{d(x, y)}\right)\]
holds for a.e. $x \in U$.
Similarly we can define the cotangent bundle $T^*X$ with the canonical metric $g_X^*$, more generally, the tensor bundle 
\[\pi: T^r_sX := \bigotimes_{i=1}^rTX \otimes \bigotimes_{i=1}^s T^*X \to X,\]
for any $r, s \in \bfZ_{\ge 0}$,
the differential $df$ and so on in an ordinary way of Riemannian geometry.
We denote by $(g_X)^r_s$ the canonical metric on $T^r_sX$ defined by $g_X$.

Moreover it was proven in \cite{Honda14a, Honda14b} that $(X, H^n)$ has the canonical 
(weakly) second order (or weak $C^{1, 1}$-) differentiable structure which is compatible with Gigli's one \cite{Gigli}.
In particular we can define the Levi-Civita connection, the Hessian of a twice differentiable function, the covariant derivative of a differentiable tensor field and so on.
We give a quick introduction of the second order differentiable structure on our setting only for reader's convenience.

In general, the singular set of a Gromov-Hausdorff limit of a sequence of Riemannian manifolds with uniform lower Ricci curvature bound has measure zero (see \cite[Theorem 2.1]{CheegerColding1}), however, even if the limit is noncollapsed, we do not know whether the singular set is closed.
In fact, Otsu-Shioya showed in \cite{OtsuShioya} that there exist a  sequence of two dimensional compact nonnegatively curved manifolds and the noncollapsed compact Gromov-Hausdorff limit $Y$ such that the singular set of $Y$ is dense in $Y$. See Example (2) in page 632 of \cite{OtsuShioya}

Cheeger-Colding proved in \cite{CheegerColding3} that there exist a countable family of Borel subsets $C_i$ of $X$ and a family of bi-Lipschitz embeddings $\phi_i$ from $C_i$ to $\bfR^n$ such that 
\[H^n\left(X \setminus \bigcup_iC_i\right)=0\]
(which means that $(X, H^n)$ is rectifiable).

Since each transition map
\[\phi_j \circ (\phi_i)^{-1}: \phi_i (C_i\cap C_j) \to \phi_j(C_i \cap C_j)\]
is bi-Lipschitz, Rademacher's theorem yields that there exists a Borel subset $D_{i, j}$ of $\phi_i (C_i\cap C_j)$ such that 
\[H^n\left( \phi_i( C_i \cap C_j)\setminus D_{i, j}\right)=0\]
and that $\phi_j \circ (\phi_i)^{-1}$ is differentiable at every $x \in D_{i, j}$ (see Section 3 in \cite{Honda14a} for the definition of differentiability of a function defined on a Borel subset of a Euclidean space).
Thus the Jacobi matrix of $\phi_j \circ (\phi_i)^{-1}$:
\[J\left(\phi_j \circ (\phi_i)^{-1}\right)(x)\]
is well defined for every $x \in D_{i, j}$.

It is known that, for any $i, j$, there exists a countable family of Borel subsets $E_{i, j, k}$ of $D_{i, j}$ such that 
\[H^n\left( D_{i, j} \setminus \bigcup_k E_{i, j, k}\right)=0\]
and that each restriction
\[J\left(\phi_j \circ (\phi_i)^{-1}\right)|_{E_{i, j, k}}\]
is a Lipschitz map. We say that the family 
\[\left\{\left((\phi_i)^{-1}(E_{i, j, k}), \phi_i\right)\right\}\]
is a second order differentiable structure of $(X, H^n)$.

It was also shown in \cite[Theorem $6.19$]{Honda14b} that 
$J$ is compatible with $g_X$ 
and that $J$ is differentiable at a.e.  $x \in X$ with $\nabla J \equiv 0$. These mean that $(g_X, J)$ gives a K\"ahler structure in some weak sense. 
We here do not discuss further detail of the above results and just refer to \cite{Honda14a, Honda14b}
because one of our main applications will be devoted to almost smooth setting and 
the assumptions above are satisfied trivially
under almost $C^2$-setting with the $C^1$-Riemannian metric, e.g. under the condition that
the Ricci curvature has two-side bound and the limit is noncollapsing. 
We shall explain $L^p$-convergence with respect to the Gromov-Hausdorff topology in section $2.2$.

We use the standard notations:
$$T_{\bfC}X:=TX \otimes_{\bfR} {\bfC} = T'X \oplus T''X,$$ 
where $T'X$ and $T''X$ are 
respectively $\sqrt{-1}$ and $-\sqrt{-1}$-eigenspaces of $J$ (note that we extended $g_X$ and $J$ in the 
$\bfC$-linear way to $T_{\bfC}X$ respectively. Define the Hermitian metric $h_X$ by  
$$h_X(u, v):=g_X(u, \overline{v}),$$ 
where $\overline{v}$ is the conjugate of $v$.
$$T_{\bfC}^*X:=T^*X \otimes_{\bfR} {\bfC} = (T^*X)' \oplus (T^*X)'',$$
where $(T^*X)'$ and 
$(T^*X)''$ are $\sqrt{-1}$ and $-\sqrt{-1}$-eigenspaces of $J^*$ which is the conjugate complex structure of $J$ and is extended $\bfC$-linearly to $T_{\bfC}^*X$. Define the Hermitian metric $h_X^*$ by 
$$h_X^*(u, v):=g_X^*(u, \overline{v}).$$
$$(T^r_s)_{\bfC}X:= \bigotimes_{i=1}^rT_{\bfC}X \otimes \bigotimes_{i=1}^s T_{\bfC}^*X \simeq  \left(\bigotimes_{i=1}^rTX \otimes \bigotimes_{i=1}^s T^*X\right) \otimes_{\bfR} \bfC.$$ 
We denote the canonical Hermitian metric on this space by $(h_X)^r_s$.

 For a Borel subset $A$ of $X$ and $p \in [1, \infty]$, let $L^p_{\bfC}((T^r_s)_{\bfC}A)$ be the set of complex valued Borel $L^p$-tensor fields on $A$.
In particular  when $r=s=0$, that is, the case of functions, we denote by $L^p_{\bfC}(A)$ 
the space of $\bfC$-valued $L^p$-functions on $A$.
For a Borel subset $A$ of $X$ and a $\bfC$-valued Borel tensor field $T$ of type $(r, s)$ on $A$ we say that $T$ is differentiable at a.e. $x \in A$ if $T^i$ is differentiable at a.e. $x \in A$ in the sense of \cite{Honda14a}, where $T=T^1 \oplus \sqrt{-1}T^2$ and $T^i$ is an $\bfR$-valued tensor field for each $i=1, 2$.
Then we put $\nabla T:=\nabla T^1 \oplus \sqrt{-1}\nabla T^2$.
Similarly for a.e. differentiable function $f$, we define $\partial f, \overline{\partial} f, \mathrm{grad}'f, \mathrm{grad}''f$ by 
$df=\partial f+\overline{\partial}f$ as $T_{\bfC}^*X = (T^*X)' \oplus (T^*X)''$ and $\grad f= \mathrm{grad}'f \oplus \mathrm{grad}''f$ as $T_{\bfC}X=T'X \oplus T''X$.

For an open subset $U$ of $X$ and $p \in (1, \infty)$ let $H^{1, p}_{\bfC}(U)$ be the completion of the space  $\mathrm{LIP}_{\mathrm{loc}, \bfC}(U)$, which is the set of $\bfC$-valued locally Lipschitz functions on $U$, with respect to the norm
\begin{align}\label{sobolevnorm}
||f||_{H^{1, p}_{\bfC}(U)}:=\left( \int_U\left(|f|^p+|d f|^p\right)\,dH^n\right)^{1/p}.
\end{align}
It is easy to check that a $\bfC$-valued function $f$ on $U$ is in $H^{1, p}_{\bfC}(U)$ if and only if $f^i \in H^{1, p}(U)$ for each $i=1, 2$, where $f=f^1+\sqrt{-1}f^2$ and $H^{1, p}(U)$ is the Sobolev space for $\bfR$-valued functions defined by the completion with respect to the norm (\ref{sobolevnorm}) 
of the space $\mathrm{LIP}_{\mathrm{loc}}(U)$ of all $\bfR$-valued locally Lipschitz functions on $U$.   
In particular for every $f \in H^{1, p}_{\bfC}(U)$, $f$ is differentiable a.e. on $U$ and $||f||_{H^{1, p}_{\bfC}}=(||f||_{L^p}^p+||d f||_{L^p}^p)^{1/p}$.

 Recall that the Levi-Civita connection and Chern connection coincide on a smooth K\"ahler manifold.
 The following  is a nonsmooth analogue of this fact.
 
\begin{proposition}\label{connection}
Let $A$ be a Borel subset of $X$ and let $V$ be a vector field on $A$ which is differentiable at a.e. $x \in A$.
If $V(x) \in T'X$ (resp. $\in T''X$) holds a.e. $x \in A$, then $\nabla V(x) \in T'X \otimes T^*_{\bfC}X$ (resp. $ \in T''X \otimes T^*_{\bfC}X$) holds a.e. $x \in A$.
\end{proposition}
\begin{proof}
The proof is standard.
See for instance page $4$ of \cite{Tian99} with $\nabla J\equiv 0$.
\end{proof}


\subsection{$L^p$-convergence on complex setting}
In \cite{Honda13, KS} the notion of $L^p$-convergence of $\bfR$-valued functions, or more generally, $\bfR$-valued tensor fields, with respect to the Gromov-Hausdorff topology was introduced.
In this section we extend this to the $\bfC$-valued case and discuss its applications.

For the reader's convenience we first recall the definition of $L^p$-convergence of $\bfR$-valued functions \cite{Honda13, KS}.
Let $p \in (1, \infty)$, let $R>0$ and let $x_i \stackrel{GH}{\to} x$, where $x_i \in X_i$ and $x \in X$.
\begin{definition}[$L^p$-convergence of $\bfR$-valued functions]\label{Lpconv}
Let $f_i$ be a sequence in $L^p(B_R(x_i))$.
\begin{enumerate}
\item[(i)] We say that $f_i$ $L^p$-converges weakly to $f \in L^p(B_R(x))$ on $B_R(x)$ if $\sup_i||f_i||_{L^p}<\infty$ and
\[\lim_{i \to \infty}\int_{B_r(y_i)}f_i\,dH^n=\int_{B_r(y)}f\,dH^n\]
hold for any sufficiently small $r>0$ and $y_i \stackrel{GH}{\to} y$, where $y_i \in B_R(x_i)$ and $y \in B_R(x)$.
\item[(ii)] We say that  $f_i$ $L^p$-converges strongly to $f \in L^p(B_R(x))$ on $B_R(x)$ if  $f_i$ $L^p$-converges weakly to $f \in L^p(B_R(x))$ on $B_R(x)$ and
\[\limsup_{i \to \infty}\int_{B_R(x_i)}|f_i|^p\,dH^n = \int_{B_R(x)}|f|^p\,dH^n.\]
\end{enumerate}
\end{definition}

Next we consider the case of vector fields:
\begin{definition}[$L^p$-convergence of $\bfR$-valued vector fields]\label{Lpreal}
Let $V_i$ be a sequence in $L^p(TB_R(x_i))$.
\begin{enumerate}
\item[(i)] We say that \textit{$V_i$ $L^p$-converges weakly to $V \in L^p(TB_R(x))$ on $B_R(x)$} if $\sup_i||V_i||_{L^p}<\infty$ and 
\[\lim_{i \to \infty}\int_{B_r(y_i)}g_{X_i}(V_i, \grad\,r_{z_i})\,dH^n=\int_{B_r(y)}g_X(V, \grad\,r_z)\,dH^n\]
holds for any sufficiently small $r>0$ and $y_i, z_i \stackrel{GH}{\to} y, z$, respectively, where $y_i, z_i \in B_R(x_i)$,  $y, z \in B_R(x)$ and $r_z$ denotes the distance function from $z$.
\item[(ii)] We say that \textit{$V_i$ $L^p$-converges strongly 
to $V \in L^p(TB_R(x))$ on $B_R(x)$} if it is an $L^p$-weak convergent sequence and 
\[\limsup_{i \to \infty}\int_{B_R(x_i)}|V_i|^p\,dH^n = \int_{B_R(x)}|V|^p\,dH^n.\]
\end{enumerate}
\end{definition}

The following proposition shows that 
the weighted version of $L^p$-convergence is equivalent to the unweighted version:
\begin{proposition}\label{Lpremark}
Let $V_i$ be a sequence in $L^p(TB_R(x_i))$.
Then $V_i$ $L^p$-converges weakly to $V$ on $B_R(x)$ if and only if $\sup_i||V_i||_{L^p}<\infty$ and 
\begin{align}\label{equiv1}
\lim_{i \to \infty}\int_{B_r(y_i)}g_{X_i}(V_i, \grad\,r_{z_i})\,dH^n_{F_i}=\int_{B_r(y)}g_X(V, \grad\,r_z)\,dH^n_F
\end{align}
hold for any sufficiently $r>0$ and $y_i, z_i \stackrel{GH}{\to} y, z$, respectively.
Moreover  $V_i$ $L^p$-converges strongly to $V$ on $B_R(x)$ if and only if $V_i$ $L^p$-converges weakly to $V$ on $B_R(x)$ and 
\begin{align}\label{equiv2}
\limsup_{i \to \infty}\int_{B_R(x_i)}|V_i|^p\,dH^n_{F_i} = \int_{B_R(x)}|V|^p\,dH^n_F
\end{align}
holds.
\end{proposition}
\begin{proof}
We only give a proof of `if' part because the proof of the converse is similar.
Suppose that $\sup_i||V_i||_{L^p}<\infty$ and (\ref{equiv1}) hold.
Then by definition, $e^{F_i}V_i$ $L^p$-converges weakly to $e^FV$ on $B_R(x)$.
Thus  $V_i$, which is equal to $e^{-F_i}(e^{F_i}V_i)$, $L^p$-converges weakly to $V=e^{-F}(e^{F}V)$ on $B_R(x)$ (c.f. \cite[Proposition $3.48$]{Honda13}).

Next suppose that $V_i$ $L^p$-converges weakly to $V$ on $B_R(x)$ and that (\ref{equiv2}) holds. 
Then by definition,  $e^{F_i/(2p)}V_i$ $L^p$-converges strongly to $e^{F/(2p)}V$ on $B_R(x)$.
Thus  $V_i=e^{-F_i/(2p)}(e^{F_i/(2p)}V_i)$ $L^p$-converges strongly
to
$$V=e^{-F/(2p)}(e^{F/(2p)}V)$$
on $B_R(x)$ (c.f. \cite[Proposition $3.70$]{Honda13}).
\end{proof}
We skip the introduction of the definition of $L^p$-convergence of general tensor fields.
However note that we can prove the equivalence as in Proposition \ref{Lpremark} for $L^p$-tensor fields.
See \cite{Honda13} for the detail.

Let $r, s \in {\bf{Z}}_{\ge 0}$.  
\begin{definition}[$L^p$-convergence of $\bfC$-valued tensor fields]\label{Lpcomplex}
Let $T_i$ be a sequence in $L^p_{\bfC}((T^r_s)_{\bfC}B_R(x_i))$.
We say that \textit{$T_i$ $L^p$-converges weakly (or strongly, respectively) to $T \in L^p_{\bfC}((T^r_s)_{\bfC}B_R(x))$ on $B_R(x)$} if $T_{i}^j$ $L^p$-converges weakly (or strongly, respectively) to $T^j$ on $B_R(x)$ for each $j =1, 2$, where $T_i=T_{i}^1+\sqrt{-1}T_{i}^2$ and $T=T^{1} + \sqrt{-1}T^{2}$.
\end{definition}
From the definition we see that many properties for $L^p$-convergence in real setting given in \cite{Honda13} can be extended canonically to the complex setting.
For example we have the following:
\begin{enumerate}
\item[(2.2a)] An $L^p$-bounded sequence has an $L^p$-weak convergent subsequence (c.f. \cite[Proposition $3.50$]{Honda13}).
\item[(2.2b)] The $L^p$-norms of an $L^p$-weak convergent sequence is lower semicontinuous (c.f. \cite[Proposition $3.64$]{Honda13}).
\item[(2.2c)] If $\sup_i||T_i||_{L^{\infty}}<\infty$, then $T_i$ $L^p$-converges weakly (or strongly, respectively) to $T$ on $B_R(x)$ for some $p \in (1, \infty)$ if and only if $T_i$ $L^p$-converges weakly (or strongly, respectively) to $T$ on $B_R(x)$ for every $p \in (1, \infty)$ (c.f. \cite[Proposition $3.69$]{Honda13}).
\item[(2.2d)] The equivalence as in Proposition \ref{Lpremark} also holds for complex valued tensor fields by the same reason.
\item[(2.2e)] Let $f$ be a complex valued Lipschitz function on $X$.
Then by \cite[Theorem $4.2$]{Honda11} there exists a sequence of $f_i \in \mathrm{LIP}_{\bf{C}}(X_i)$ 
with 
$$\sup_i||df_i||_{L^{\infty}}<\infty$$ 
such that $f_i, d f_i$ $L^2$-converge strongly to $f, d f$ on $X$, respectively.
\item[(2.2f)] Let $T \in L^p_{\bfC}((T^r_s)_{\bfC}B_R(x))$. Then there exists a sequence of $T_i \in  L^p_{\bfC}((T^r_s)_{\bfC}B_R(x_i))$ such that $T_i$ $L^p$-converges strongly to $T$ on $B_R(x)$.
\end{enumerate}

The following Rellich Lemma plays a key role in establishing the spectral convergence of weighted ($\overline{\partial}$-) Laplacian:
\begin{theorem}[Rellich compactness]\label{Rellich}
Let $f_i$ be a sequence in $H^{1, p}_{\bf{C}}(B_R(x_i))$ with $\sup_i||f_i||_{H^{1, p}_{\bf{C}}}<\infty$.
Then there exist a subsequence $f_{i(j)}$ and $f \in H^{1, p}_{\bf{C}}(B_R(x))$ such that $f_{i(j)}$ $L^p$-converges strongly to $f$ on $B_R(x)$ and that $d f_{i(j)}$ $L^p$-converges weakly to $d f$ on $B_R(x)$. 
\end{theorem}
\begin{proof}
This is a direct consequence of the real version shown in \cite[Theorem $4.9$]{Honda13}.
\end{proof}
For every $l \in \{1, \ldots, r+s\}$ let $J^l$ be the complex structure on $(T^r_s)_xX$ defined by
\begin{align*}
&J^l(v_1\otimes \cdots \otimes v_r \otimes v_{r+1}^*\otimes \cdots \otimes v_{r+s}^*) \\
&:=
\begin{cases} v_1 \otimes \cdots \otimes v_{l-1} \otimes Jv_l \otimes v_{l+1} \cdots \otimes v_r \otimes v_{r+1}^*\otimes \cdots \otimes v_{r+s}^* \,\,\,\,\,\mathrm{if}\,l \le r, \\
v_1 \otimes \cdots \otimes v_r \otimes v_{r+1}^*\otimes \cdots  \otimes v_{l-1}^* \otimes J^*v_l^* 
\otimes v_{l+1}^* \cdots \otimes v_{r+s}^* \,\,\,\,\,\mathrm{if}\,l \ge r+1.
\end{cases}
\end{align*}
\begin{proposition}\label{convergencedecomp}
Let $T_i$ be a sequence in $L^p_{\bfC}((T^r_s)_{\bfC}B_R(x_i))$ and let $T \in L^p_{\bf{C}}((T^r_s)_{\bf{C}}B_R(x))$.
Then the following are equivalent.
\begin{enumerate}
\item[(1)] $T_i$ $L^p$-converges weakly (or strongly, respectively) to $T$ on $B_R(x)$.
\item[(2)] For every $l \in \{1, \ldots, r+s\}$, $T_i'$ and $T_i''$ $L^p$-converge weakly (or strongly, respectively) to $T'$ and $T''$ on $B_R(x)$, respectively, where $T_i=T_i^{'} \oplus T_i^{''}$ and $T=T{'} \oplus T^{''}$ with respect to the decompositions by $\pm\sqrt{-1}$-eigenspaces of $J_{i}^l$ and $J^l$, respectively.
\item[(3)] For some $l \in \{1, \ldots, r+s\}$, $T_i^{'}$ and $T_i^{''}$ $L^p$-converge weakly (or strongly, respectively) to $T^{'}$ and $T^{''}$ on $B_R(x)$, respectively.
\end{enumerate}
\end{proposition}
\begin{proof}
Since $J_{i}^l$ $L^2$-converges strongly to $J^{l}$ on $X$ with $\sup_i||J_i^l||_{L^{\infty}}<\infty$, the assertion follows from \cite[Propositions $3.48$ and $3.70$]{Honda13} and equalities
\[T_i^{'}=\frac{1}{2}(T_i-\sqrt{-1}J_{i}^lT_i),\,T_i^{''}=\frac{1}{2}(T_i+\sqrt{-1}J_{i}^lT_i).\] 
\end{proof}
\begin{remark}\label{decompremark}
It is a direct consequence of Proposition \ref{convergencedecomp} that the \textit{type} of tensor fields is preserved with respect to the $L^p$-weak convergence.
For example the $L^p$-weak limit of a sequence of $(q, r)$-forms is also a $(q, r)$-form.
\end{remark}
\begin{proposition}\label{L2norm}
Let $f$ and $g$ be in the set $\mathrm{LIP}_{\bfC}(X)$ of all Lipschitz functions on $X$.
Then 
\begin{align}\label{innerproduct}
\int_Xh^*_X(df, dg)dH^n=2\int_Xh_X^*(\barpartial f, \barpartial g)dH^n=2\int_Xh^*_X(\partial f, \partial g)dH^n.
\end{align}
In particular,
\[\int_X|d f|^2\,dH^n=2\int_X|\partial f|^2\,dH^n=2\int_X|\overline{\partial} f|^2\,dH^n\]
and
\[\int_X|d f|^2\,dH_F^n \stackrel{L}{\asymp} \int_X|\partial f|^2\,dH^n_F \stackrel{L}{\asymp} \int_X|\overline{\partial} f|^2\,dH_F^n,\]
where for any nonnegative real numbers $a, b$, $a \stackrel{L}{\asymp} b$ means that there exists a positive constant $C:=C(L)>1$ depending only on $L$ such that $C^{-1}b\le a \le Cb$ holds, $L$ being
the constant in (2.1d).
\end{proposition}
\begin{proof}
By (2.2e), there exist sequences $f_i$ and $g_i \in \mathrm{LIP}_{\bfC}(X_i)$ such that 
$$\sup_i\left(||d f_i||_{L^{\infty}}+||d g_i||_{L^{\infty}}\right)<\infty$$ 
and that $f_i, d f_i, g_i$ and $dg_i$ $L^2$-converge strongly to $f, df, g$ and $dg$ on $X$ respectively.
By the smoothing via the heat flow (c.f. \cite{AGS, Grig}) without loss of generality we can assume $f_i, g_i \in C^{\infty}_{\bfC}(X_i)$  for every $i<\infty$, where $C^{\infty}_{\bfC}(X_i)$ is the set of $\bfC$-valued smooth functions on $X_i$.
Since $\Delta=2\Delta_{\overline{\partial}}$ holds on smooth setting, we have
\begin{eqnarray}\label{1}
\int_{X_i}h^*_{X_i}(df_i, dg_i)\,dH^n &=& \int_{X_i}(\Delta f_i)\overline{g_i}\,dH^n\\
&=&2\int_{X_i}(\Delta_{\overline{\partial}}f_i)\overline{g_i}\,dH^n\nonumber\\
&=&2\int_{X_i}h^*_{X_i}(\barpartial f_i, \barpartial g_i)\,dH^n.\nonumber
\end{eqnarray}
Thus since Proposition \ref{convergencedecomp} yields that $\overline{\partial}f_i$ and $\barpartial g_i$ $L^2$-converge strongly to $\overline{\partial}f$ and $\barpartial g$ on $X$ respectively by letting $i \to \infty$ in (\ref{1}), we have
\[\int_Xh^*_X(df, dg)\,dH^n=2\int_Xh_X^*(\barpartial f, \barpartial g)\,dH^n.\]
Similarly we have 
\[\int_Xh^*_X(df, dg)\,dH^n=2\int_Xh_X^*(\partial f, \partial g)\,dH^n.\] 
This completes the proof.
\end{proof}
\begin{remark}
By Proposition \ref{L2norm} the completion of $\mathrm{LIP}_{\bfC}(X)$ with respect to the norm
\[\left(\int_X\left(|f|^2+|\overline{\partial}f|^2\right)\,dH^n_F\right)^{1/2}\]
coincides with $H^{1, 2}_{\bfC}(X)$ (however the norms are different).
\end{remark}
\begin{corollary}\label{localconv}
Let $f_i$ be a sequence of $H^{1, 2}_{\bfC}(B_R(x_i))$ with
\[\sup_i\left( \int_{B_R(x_i)}(|f_i|^2+|\overline{\partial}f_i|^2)\,dH^n \right)<\infty,\]
and let $f$ be the $L^2$-weak limit of them on $B_R(x)$.
Then for every $r<R$ we see that $f|_{B_r(x)} \in H^{1, 2}_{\bfC}(B_r(x))$, that $f_i$ $L^2$-converges strongly to $f$ on $B_r(x)$ and that $df_i$ $L^2$-converges weakly to $df$ on $B_r(x)$.
Moreover if $\overline{\partial}f_i$ $L^2$-converges strongly to $\overline{\partial}f$ on $B_s(x)$ for some $s<R$, then $df$ $L^2$-converges strongly to $df$ on $B_r(x)$ for every $r<s$.
\end{corollary}
\begin{proof}
This follows directly from Theorem \ref{Rellich} and the following claim:
\begin{claim}\label{bound} 
Let $f \in H^{1, 2}_{\bfC}(B_R(x))$ with 
\[\int_{B_R(x)}(|f|^2+|\overline{\partial}f|^2)\,dH^n\le \hat{L}.\]
Then for every $r<R$ we have
\[\int_{B_r(x)}|df|^2\,dH^n\le C(r, R, \hat{L}).\]
\end{claim}

\noindent
The proof is as follows:
Let $r<R$ and $u := (r+R)/2$. Let $g_{r, R}$ be the Lipschitz function on $\bfR$ defined by 
\begin{align*}
g_{r, R}(t):=
\begin{cases} 1 \,\,\,\,\,\mathrm{if}\,t \le r, \\
\frac{u-t}{u-r} \,\,\,\,\,\mathrm{if}\, r \le t \le u, \\
0 \,\,\,\,\,\mathrm{if}\, u \le t,
\end{cases}
\end{align*}
and let $G=G^x_{r, R}$ be the Lipschitz function on $X$ defined by $G(y):=g_{r, R}(d_X(x, y))$.
Then since $|\nabla G| \le C(r, R)$, we have $G f \in H^{1, 2}_{\bfC}(X)$ and 
\[\int_X|\overline{\partial}(G f)|^2\,dH^n\le C(r, R, \hat{L}),\]
Proposition \ref{L2norm} gives
\[\int_{B_r(x)}|df|^2\,dH^n\le \int_X|d(G f)|^2\,dH^n =2 \int_X|\overline{\partial}(G f)|^2\,dH^n\le C(r, R, \hat{L}).\]
This completes the proof of Claim \ref{bound}.

Claim \ref{bound} with Theorem \ref{Rellich} yields that  for every $r<R$ we see that $f|_{B_r(x)} \in H^{1, 2}_{\bfC}(B_r(x))$, that $f_i$ $L^2$-converges strongly to $f$ on $B_r(x)$ and that $df_i$ $L^2$-converges weakly to $df$ on $B_r(x)$.

Next we suppose that $\overline{\partial}f_i$ $L^2$-converges strongly to $\overline{\partial}f$ on $B_s(x)$ for some $s<R$.
Let $r<s$, let $G_i:=G_{r, s}^{x_i}$ and let $G:=G_{r, s}^x$.
Then since $dr_{x_i}$ $L^2$-converges strongly to $dr_x$ on $X$ (c.f. \cite[Proposition $3.44$]{Honda13}), we see that $G_i, dG_i$ $L^2$-converge strongly to $G, dG$ on $X$, respectively.
Note that  $G_i f_i \in H^{1, 2}_{\bfC}(X_i)$ and $G f \in H^{1, 2}_{\bfC}(X)$ hold. 
Proposition \ref{L2norm} and the assumption give
\begin{align*}
&\lim_{i \to \infty}\int_{X_i}|d(G_if_i)|^2\,dH^n\\
&=2\lim_{i \to \infty}\int_{X_i}|\overline{\partial}(G_if_i)|^2\,dH^n\\
&=2\lim_{i \to \infty}\int_{B_r(x_i)}\left( |f_i|^2|\overline{\partial}G_i|^2+\overline{f_i}G_ih_{X_i}(\overline{\partial}f_i, \overline{\partial}G_i)\right.\\
& \qquad\qquad\qquad\qquad\qquad  \left.\qquad +\ f_iG_ih_{X_i}(\overline{\partial}G_i, \overline{\partial}f_i)+|G_i|^2|\overline{\partial}f_i|^2\right)\,dH^n \\
&=2\int_{B_r(x)}\left( |f|^2|\overline{\partial}G|^2+\overline{f}Gh_{X}(\overline{\partial}f, \overline{\partial}G)+fGh_{X}(\overline{\partial}G, \overline{\partial}f)+|G|^2|\overline{\partial}f|^2\right)\,dH^n \\
&=2\int_X|\overline{\partial}(Gf)|^2\,dH^n\\
&=\int_X|d(Gf)|^2\,dH^n.
\end{align*}
Thus $d(G_if_i)$ $L^2$-converges strongly to $d(Gf)$ on $X$.
By restricting this on $B_r(x)$ we have the assertion.
\end{proof}

\section{Weighted Laplacian on the limit space}
\subsection{Weighted Laplacian and weighted $\overline{\partial}$-Laplacian}
From now on we consider the weighted measure:
\[dH^n_{F_i}:=e^{F_i}\,dH^n.\] 
As stated in 2.1, this measure converges to $e^F dH^n$ in our setting. 
Let $U$ be an open subset of $X$.
\begin{remark}\label{weightremark}
The completion of $\mathrm{LIP}_{\bfC}(U)$ with respect to the weighted norm
\[\left( \int_U\left(|f|^p+|d f|^p\right)\,dH^n_F\right)^{1/p}.\]
coincides with $H^{1, p}_{\bfC}(U)$ (but the norms differ)  because $0<C_1(L) \le e^F\le C_2(L)<\infty$ holds, where $C_i(L)$ is a positive constant depending only on $L$ in (2.1d).
\end{remark}
\begin{definition}[Weighted Laplacian]
Let $\mathcal{D}^2_{\bfC}(\Delta^F, U)$ be the set of $f \in H^{1, 2}_{\bfC}(U)$ such that there exists $g \in L^2_{\bfC}(U)$ satisfying
\begin{align}\label{laplaciandefinition}
\int_Uh_X^\ast (d f, d \phi)\ dH^n_F=\int_U g\overline{\phi}\ dH^n_F
\end{align}
for every $\phi \in \mathrm{LIP}_{c, \bfC}(U)$, where $\mathrm{LIP}_{c, \bfC}(U)$ is the space of $\bfC$-valued Lipschitz functions on $U$ with compact support.
Since $g$ is unique we denote it by $\Delta^{F}f$.
\end{definition}
If $F\equiv 0$, then we put $\Delta:=\Delta^0$.

Recall that we can define $\Delta^F$ as real operator as follows:
Let $\mathcal{D}^2(\Delta^F, U)$ be the set of $f \in H^{1, 2}(U)$ such that there exists a real valued $L^2$-function $g \in L^2(U)$ satisfying (\ref{laplaciandefinition}) for every $\phi \in \mathrm{LIP}_c(U)$, where $L^2(U)$ is the set of $\bfR$-valued Borel $L^2$-functions on $U$ and  $\mathrm{LIP}_c(U)$ is the set of $\bfR$-valued Lipschitz functions on $U$ with compact support. 
In this case since $g$ is unique we denote it by $\Delta^F_{\bfR}f$, or $\Delta^Ff$ because the following proposition holds:
\begin{proposition}\label{Laplacianreal}
Let $f=f^1 +\sqrt{-1}f^2$ be a function on $U$.
We see that $f \in \mathcal{D}^2_{\bfC}(\Delta^F, U)$ holds if and only if $f^i \in \mathcal{D}^2(\Delta^F, U)$ holds for each $i=1, 2$.
Moreover if  $f \in \mathcal{D}^2_{\bfC}(\Delta^F, U)$, then $\Delta^Ff=\Delta^F_{\bfR}f^1 + \sqrt{-1}\Delta^F_{\bfR}f^2$.
\end{proposition}
\begin{proof}
It is a direct consequence to substitute $f=f^1 + \sqrt{-1}f^2$ and $\phi=\phi^1+\sqrt{-1}\phi^2$ in (\ref{laplaciandefinition}).
\end{proof}
\begin{corollary}
Let $f \in  \mathcal{D}^2_{\bf{C}}(\Delta^F, U)$.
Then $\overline{f} \in \mathcal{D}^2_{\bf{C}}(\Delta^F, U)$ with $\Delta^F\overline{f}=\overline{\Delta^Ff}$.
Moreover we have the following:
\begin{enumerate}
\item[(1)] The eigenvalues of $\Delta^F$ are nonnegative real numbers.
\item[(2)] For any $f \in \mathcal{D}^2_{\bfC}(\Delta^F, X)$ and $\lambda \ge 0$, $f$ is a $\lambda$-eigenfunction of $\Delta^F$ if and only if $f^i$ is a  $\lambda$-eigenfunction of $\Delta^F$ for each 
$i=1, 2$,
\end{enumerate}
\end{corollary}
\begin{proof}
Proposition \ref{Laplacianreal} yields $\overline{f} \in \mathcal{D}^2_{\bfC}(\Delta^F, U)$ with $\Delta^F\overline{f}=\Delta^Ff^1-\sqrt{-1}\Delta^Ff^2=\overline{\Delta^Ff}$.

Let $f$ be a $\lambda$-eigenfunction of $\Delta^F$.
Since
\[0 \le \int_Xh_X^\ast (d f, d f)\,dH^n_F=\int_X(\Delta^Ff)\overline{f}\,dH^n_F=\lambda \int_X|f|^2\,dH^n_F,\]
$\lambda$ is a nonnegative real number.
Therefore
\[\Delta^F f^1=\Delta^F \left(\frac{f+\overline{f}}{2}\right)=\frac{\Delta^Ff+\overline{\Delta^Ff}}{2}=\lambda f^1.\]
Similarly we have $\Delta^Ff^2=\lambda f^2$.
This completes the proof.
\end{proof}
We now give the definition of weighted $\overline{\partial}$-Laplacian.
\begin{definition}[Weighted $\overline{\partial}$-Laplacian]\label{dlaplacian}
Let $\mathcal{D}^2_{\bfC}(\Delta^F_{\overline{\partial}}, U)$ be the set of $f \in H^{1, 2}_{\bfC}(U)$ such
 that there exists $g \in L^2_{\bfC}(U)$ satisfying
\begin{align}\label{weightedlap}
\int_Uh_X^*(\overline{\partial}f, \overline{\partial}\phi)\,dH^n_F=\int_Ug\overline{\phi}\,dH^n_F
\end{align}
for every $\phi \in \mathrm{LIP}_{c, \bfC}(U)$.
Since $g$ is unique we denote it by $\Delta^F_{\overline{\partial}}f$.
\end{definition}
If $F \equiv 0$, then we put $\Delta_{\overline{\partial}}:=\Delta^0_{\overline{\partial}}$.

The following relationship between $\Delta$ and $\Delta_{\barpartial}$ is well known on smooth setting:
\begin{proposition}\label{Laplacianequal}
We have $\mathcal{D}^2_{\bfC}(\Delta, U)=\mathcal{D}^2_{\bfC}(\Delta_{\overline{\partial}}, U)$ for every open subset $U$ of $X$.
Moreover for every $f \in \mathcal{D}^2_{\bfC}(\Delta, U)$ we have
\[\Delta f=2\Delta_{\overline{\partial}}f.\]
\end{proposition}
\begin{proof} 
This is a direct consequence of the following:
\begin{claim}
Let $f \in H^{1, 2}_{\bfC}(U)$ and let $g \in \mathrm{LIP}_{c, \bfC}(U)$.
Then
\[\int_Uh^*_X(df, dg)dH^n=2\int_Uh_X^*(\barpartial f, \barpartial g)dH^n.\]
\end{claim}
The proof is as follows.
There exists $\phi \in \mathrm{LIP}_c(U)$ such that $\phi|_{\mathrm{supp}\,g}\equiv 1$.
Then since it is easy to check that $\phi f \in H^{1, 2}(X)$, (\ref{innerproduct}) gives
\begin{align*}
\int_Uh_X^*(df, dg)dH^n&=\int_Xh^*_X(d(\phi f), dg)dH^n\\
&=2\int_Xh^*_X(\barpartial (\phi f), \barpartial g)dH^n=2\int_Uh^*_X(\barpartial f, \barpartial g)dH^n.
\end{align*}
\end{proof}

The eigenvalues of $\Delta^F_{\overline{\partial}}$ on $X$ are also nonnegative real numbers because
\[\int_X(\Delta^F_{\overline{\partial}}u)\overline{v}\,dH^n_F=\int_Xh_X^*(\overline{\partial}u, \overline{\partial}v)\,dH^n_F=\int_Xu\overline{\Delta^F_{\overline{\partial}}v}\,dH^n_F\]
holds for any $u, v \in \mathcal{D}^2_{\bfC}(\Delta^F_{\overline{\partial}}, X)$.

\begin{proposition}\label{lapdom}
Assume that $H^n(X \setminus U)=0$ and that
the inclusion
\[H^{1, 2}_{c}(U) \hookrightarrow H^{1, 2}(X)\]
is isomorphic, where $H^{1, 2}_{c}(U)$ is the closure of $\mathrm{LIP}_c(U)$ in $H^{1, 2}(X)$.
Let $f \in H^{1, 2}_{\bfC}(X)$ with $f|_U \in \mathcal{D}^2_{\bfC}(\Delta_{\overline{\partial}}, U)$ (or $f|_U \in \mathcal{D}^2_{\bfC}(\Delta^F, U)$, respectively). 
Then $f \in \mathcal{D}^2_{\bfC}(\Delta_{\overline{\partial}}, X)$ (or $f \in \mathcal{D}^2_{\bfC}(\Delta^F, X)$, respectively).
\end{proposition}
\begin{proof}
We only give a proof in the case of $\Delta^F_{\overline{\partial}}$.

Let $g \in \mathrm{LIP}_{\bfC}(X)$.
By the assumption, there exists a sequence $g_i \in \mathrm{LIP}_{c, \bfC}(U)$ such that $g_i \to g$ in $H^{1, 2}_{\bfC}(X)$.
Then since
\[\int_Xh_X(\overline{\partial} f, \overline{\partial} g_i)\,dH^n_F=\int_X(\Delta^F_{\overline{\partial}}f)\overline{g_i}\,dH^n_F,\]
letting $i \to \infty$ shows that $f \in \mathcal{D}^2_{\bfC}(\Delta_{\overline{\partial}}^F, X)$.
\end{proof}
\begin{remark}
In general, if $\mathrm{dim}_H(X \setminus U)<n-2$, then 
the inclusion
\[H^{1, 2}_{c}(U) \hookrightarrow H^{1, 2}(X)\]
is isomorphic.
See for instance \cite[Theorem $4.6$]{KKM}, \cite[Theorem $4.13$]{KM} and \cite[Theorem $4.8$]{Shanm}.
Moreover if $X \setminus U$ satisfies a good regularity (e.g. it is a submanifold), then the above isometry hold even if $\mathrm{dim}_H(X \setminus U)=n-2$.
\end{remark}
We end this section by giving a relationship between $\Delta^F, \Delta^F_{\overline{\partial}}$ and $\Delta, \Delta_{\overline{\partial}}$, respectively, which are well-known on smooth setting.
\begin{proposition}\label{Laplacianexplicite}
Suppose $F|_U \in H^{1, 2}(U)$.
Then for every $f \in H^{1, 2}_{\bfC}(U)$,  we have the following equivalence:
\begin{enumerate}
\item[(1)] If $g_X^\ast (d f, d F) \in L^2_{\bfC}(U)$, then $f \in \mathcal{D}^2_{\bfC}(\Delta^F, U)$ holds if and only if $f \in \mathcal{D}^2_{\bfC}(\Delta, U)$ holds.
In this case $\Delta^Ff=\Delta f-g_X^\ast (d f, d F)$.
\item[(2)] If $h_X(\overline{\partial} f, \overline{\partial}F) \in L^2_{\bfC}(U)$, then $f \in \mathcal{D}^2_{\bfC}(\Delta_{\overline{\partial}}^F, U)$ holds if and only if $f \in \mathcal{D}^2_{\bfC}(\Delta_{\overline{\partial}}, U)$ holds.
In this case $\Delta^F_{\overline{\partial}}f=\Delta_{\overline{\partial}}f-h_X^\ast (\overline{\partial}f, \overline{\partial}F)$.
\end{enumerate}
\end{proposition}
\begin{proof}
We give a proof of `only if' part of (2) because the proofs of the other cases are similar. 

Let  $f \in \mathcal{D}^2_{\bfC}(\Delta_{\overline{\partial}}^F, U)$ and let $\phi \in \mathrm{LIP}_{c, \bfC}(U)$.
Since $e^{-F}\phi \in \mathrm{LIP}_{c, \bfC}(U)$, we have
\begin{align*}
\int_U(\Delta^F_{\overline{\partial}}f)\overline{\phi}\,dH^n &= \int_U(\Delta^F_{\overline{\partial}}f)\overline{e^{-F}\phi}\,dH^n_F \\
&=\int_Uh_X^*\left(\overline{\partial}f, \overline{\partial}(e^{-F}\phi)\right)\,dH^n_F \\
&=\int_Uh_X^*\left(\overline{\partial}f, -e^{-F}\phi \overline{\partial}F+ e^{-F}\overline{\partial} \phi\right)e^F\,dH^n \\
&=-\int_Uh_X^*\left(\overline{\partial}f,  \overline{\partial}F\right)\overline{\phi}\,dH^n+\int_Uh_X^*\left(\overline{\partial}f,  \overline{\partial}\phi \right)\,dH^n.
\end{align*}
Thus
\[\int_Uh_X^*\left(\overline{\partial}f,  \overline{\partial}\phi \right)\,dH^n=\int_U\left(\Delta^F_{\overline{\partial}}f+h_X^*\left(\overline{\partial}f,  \overline{\partial}F\right)\right) \overline{\phi}\,dH^n.\]
This completes the proof.
\end{proof}
\begin{remark}
Similarly we can define the weighted $\partial$-Laplacian, $\Delta^F_{\partial}$, and prove similar results above.
By combining Remark 2.10 with Theorem \ref{Rellich} we see that the spectrums of $\Delta^F_{\overline{\partial}}$, $\Delta_{\partial}^F$ and $\Delta^F$ are discrete and unbounded, and that each eigenspace is finite dimensional.

\end{remark}
\begin{remark}\label{L1lap}
For any $q \in (1, \infty)$ and $p \in [1, \infty)$, let $\mathcal{D}^{q, p}_{\bfC}(\Delta^F_{\barpartial}, X)$ be the set of $f \in H^{1, q}(U)$ such that there exists $g \in L^p_{\bfC}(U)$ such that (\ref{weightedlap}) holds for every $\phi \in \mathrm{LIP}_{\bfC}(X)$.
Since $g$ is unique, we also denote it by $\Delta^F_{\barpartial}f$.
Then by the proof of Proposition \ref{Laplacianexplicite}, for every $f \in \mathcal{D}^2_{\bfC}(\Delta_{\barpartial}^{F}, U)$ (or $f \in \mathcal{D}^2_{\bfC}(\Delta_{\overline{\partial}}, U)$, respectively), we have $f \in \mathcal{D}^{2, 1}_{\bfC}(\Delta_{\overline{\partial}}, U)$  (or $f \in \mathcal{D}^{2, 1}_{\bfC}(\Delta_{\barpartial}^{F}, U)$, respectively) with $\Delta^F_{\overline{\partial}}f=\Delta_{\overline{\partial}}f-h_X^\ast (\overline{\partial}f, \overline{\partial}F)$.
Note that $\mathcal{D}^{2, 2}_{\bfC}(\Delta^F_{\barpartial}, U)=\mathcal{D}^2_{\bfC}(\Delta^F_{\barpartial}, U)$.
\end{remark}

\subsection{Spectral convergence}
From now on we will discuss the behavior of $\Delta^F_{\overline{\partial}}$ with respect to the Gromov-Hausdorff topology.
We first show the spectral convergence. 
Note that by Proposition \ref{L2norm} the smallest eigenvalue of $\Delta^{F}_{\overline{\partial}}$ on $X$ is $0$.
\begin{proposition}\label{spectralconv}
For every $k \ge 1$ we have
\[\lim_{i \to \infty}\lambda_k(\Delta^{F_i}_{\overline{\partial}}, X_i)=\lambda_k(\Delta^{F}_{\overline{\partial}}, X),\]
where $\lambda_k(\Delta^F_{\overline{\partial}}, X)$ denotes the $k$-th positive eigenvalue of $\Delta^F_{\overline{\partial}}$ on $X$
 counted with multiplicity.
\end{proposition}
\begin{proof}
This is a direct consequence of Theorem \ref{Rellich} and min-max principle.
However we give a proof in the case $k=1$ for reader's convenience (c.f. \cite[Theorem $1.5$]{Honda14b}).

We first prove the upper semicontinuity of $\lambda_1(\Delta^{F_i}_{\overline{\partial}}, X_i)$.
Recall that
\begin{align*}
\lambda_1(\Delta^F_{\overline{\partial}}, X)=\inf_{f}\frac{\int_X|\barpartial f|^2\,dH^n_F}{\int_X|f|^2\,dH^n_F}
\end{align*}
where $f$ runs over all nonconstant Lipschitz functions with
\begin{align}\label{2}
\int_Xf\,dH^n_F=0.
\end{align}
Let $f$ be a nonconstant complex valued Lipschitz function on $X$ with (\ref{2}).
Then by (2.2e), there exists a sequence of $f_i \in \mathrm{LIP}_{\bfC}(X_i)$ with 
\[\int_{X_i}f_i\,dH^n_{F_i}=0\]
such that $f_i, d f_i$ $L^2$-converge strongly to $f, d f$ on $X$, respectively.
Thus
\[\limsup_{i \to \infty}\lambda_1(\Delta^{F_i}_{\overline{\partial}}, X_i)\le \lim_{i \to \infty}\frac{\int_{X_i}|\barpartial f_i|^2\,dH^n_{F_i}}{\int_{X_i}|f_i|^2\,dH^n_{F_i}}=\frac{\int_X|\barpartial f|^2\,dH^n_F}{\int_X|f|^2\,dH^n_F}.\]
Since $f$ is arbitrary, we have 
\[\limsup_{i \to \infty}\lambda_1(\Delta^{F_i}_{\overline{\partial}}, X_i) \le \lambda_1(\Delta^{F}_{\overline{\partial}}, X).\]
Next we prove the lower semicontinuity.
Let $f_i$ be a sequence in $\mathcal{D}^2_{\bfC}(\Delta^{F_i}_{\overline{\partial}}, X_i)$ with
$$\Delta^{F_i}_{\overline{\partial}}f_i=\lambda_1(\Delta^{F_i}_{\overline{\partial}}, X_i)f_i$$
and
$$\int_{X_i}|f_i|^2\,dH^n_{F_i}=1.$$
Then it follows from
\[\int_{X_i}|\barpartial f_i|^2\,dH^n_{F_i}=\int_{X_i}(\Delta^{F_i}_{\overline{\partial}}f_i)\overline{f_i}\,dH^n_{F_i}=\lambda_1(\Delta^{F_i}_{\overline{\partial}}, X_i)\]
and the upper semicontinuity of $\lambda_1(\Delta^{F_i}_{\overline{\partial}}, X_i)$
that $\sup_i||f_i||_{H^{1, 2}_{\bfC}}<\infty$ holds.
Thus by Theorem \ref{Rellich} without loss of generality we can assume that there exists $f \in H^{1, 2}_{\bfC}(X)$ such that $f_i$ $L^2$-converges strongly to $f$ on $X$ and that $d f_i$ $L^2$-converges weakly to $d f$ on $X$.
In particular Proposition \ref{convergencedecomp} yields that $\overline{\partial} f_i$ $L^2$-converges weakly to $\overline{\partial} f$ on $X$.
Thus by the lower semicontinuity of the $L^2$-norms of an $L^2$-weak convergent sequence, we have 
\begin{eqnarray*}
\liminf_{i \to \infty}\lambda_1(\Delta^{F_i}_{\overline{\partial}}, X_i)&=&\liminf_{i \to \infty} \int_{X_i}|\overline{\partial}f_i|^2\,dH^n_{F_i}\\
&\ge& \int_X|\overline{\partial}f|^2\,dH^n_F \\
&\ge& \lambda_1(\Delta^{F}_{\overline{\partial}}, X),
\end{eqnarray*}
where we used 
$$
\int_X|f|^2\,dH^n_F=\lim_{i \to \infty}\int_{X_i}|f_i|^2\,dH^n_{F_i}=1
$$
and
$$
\int_Xf\,dH^n_F=\lim_{i \to \infty}\int_{X_i}f_i\,dH^n_{F_i}=0.
$$
This completes the proof.
\end{proof}
\begin{proposition}\label{Laplacianweakconv}
Let $f$ be the $L^2$-weak limit on $X$ of a sequence of $f_i \in \mathcal{D}^2_{\bfC}(\Delta^{F_i}_{\overline{\partial}}, X_i)$ with 
\[\sup_i(||f_i||_{H^{1, 2}_{\bfC}}+||\Delta^{F_i}_{\overline{\partial}}f_i||_{L^2})<\infty.\]
Then we have the following:
\begin{enumerate}
\item[(1)] $f \in \mathcal{D}^2_{\bfC}(\Delta^F_{\overline{\partial}}, X)$.
\item[(2)] $f_i,\ d f_i$ $L^2$-converge strongly to $f,\ d f$ on $X$, respectively.
\item[(3)] $\Delta^{F_i}_{\overline{\partial}}f_i$ $L^2$-converges weakly to $\Delta^F_{\overline{\partial}}f$ on $X$.
\end{enumerate}
\end{proposition}
\begin{proof}
By Theorem \ref{Rellich} we see that $f \in H^{1, 2}_{\bfC}(X)$, that $f_i$ $L^2$-converges strongly to $f$ on $X$ and that $d f_i$ $L^2$-converges weakly to $d f$ on $X$.
By the compactness of $L^2$-weak convergence, without loss of generality we can assume that there exists the $L^2$ weak limit $G \in L^2_{\bfC}(X)$ of $\Delta^{F_i}_{\overline{\partial}}f_i$.
Let $\phi \in \mathrm{LIP}_{\bfC}(X)$. By \cite[Theorem $4.2$]{Honda11} there exists a sequence $\phi_i \in \mathrm{LIP}_{\bfC}(X_i)$ such that $\phi_i, d\phi_i$ $L^2$-converge strongly to $\phi, d\phi$ on $X$, respectively.
Proposition \ref{convergencedecomp} yields that $\overline{\partial}f_i$ $L^2$-converges weakly to $\overline{\partial}f$ on $X$ and that $\overline{\partial}\phi_i$ $L^2$-converges strongly to $\overline{\partial}\phi$ on $X$.
Since
\[\int_{X_i}h_{X_i}(\overline{\partial}f_i, \overline{\partial}\phi_i)\,dH^n_{F_i}=\int_{X_i}(\Delta^{F_i}_{\overline{\partial}}f_i)\overline{\phi_i}\,dH^n_{F_i},\]
by letting $i \to \infty$ we have
\[\int_X h_X(\overline{\partial}f, \overline{\partial}\phi)\,dH^n_F = \int_X G\overline{\phi}\,dH^n_F.\]
This gives (1) and (3).

On the other hand 
\begin{eqnarray*}
\lim_{i \to \infty}\int_{X_i}|\overline{\partial}f_i|^2\,dH^n_{F_i} &=& \lim_{i \to \infty}\int_{X_i}(\Delta^{F_i}_{\overline{\partial}}f_i)\overline{f_i}\,dH^n_{F_i}\\
&=& \int_{X}(\Delta^F_{\overline{\partial}}f)\overline{f}\,dH^n_F\\
&=& \int_X|\overline{\partial}f|^2\,dH^n_F.
\end{eqnarray*}
Thus by Proposition \ref{Lpremark}, $\overline{\partial}f_i$ $L^2$-converges strongly to $\overline{\partial}f$ on $X$.
Therefore
Corollary \ref{localconv} gives (2). 
\end{proof}

\begin{proposition}\label{Poincare}
For any $r\le R$ and $f \in H^{1, 2}_{\bfC}(B_r(x))$ we have
\begin{align}\label{poincareineq}
&\frac{1}{H^n_F(B_r(x))}\int_{B_r(x)}\left| f-\frac{1}{H^n_F(B_r(x))}\int_{B_r(x)}f\,dH^n_F\right|^2\,dH^n_F \\
&\le C(n, K, R, L)\frac{r^2}{H^n_F(B_r(x))}\int_{B_r(x)}|d f|^2\,dH^n_F. \nonumber
\end{align}
\end{proposition}
\begin{proof}
It is a direct consequence of \cite[Theorem $2.15$]{CheegerColding3} that (\ref{poincareineq}) holds if $F \equiv 0$.
Since $e^{-L} \le e^F \le e^L$ and the left hand side of (\ref{poincareineq}) is equal to
\[\inf_{c \in \bfC}\left(\frac{1}{H^n_F(B_r(x))}\int_{B_r(x)}| f-c|^2\,dH^n_F\right),\]
we have
\begin{align*}
&\inf_{c \in \bfC}\left(\frac{1}{H^n_F(B_r(x))}\int_{B_r(x)}| f-c|^2\,dH^n_F\right) \\
&\le C(n, K, R, L)\inf_{c \in \bfC}\left(\frac{1}{H^n(B_r(x))}\int_{B_r(x)}| f-c|^2\,dH^n\right)\\
&\le C(n, K, R, L)\frac{r^2}{H^n(B_r(x))}\int_{B_r(x)}|d f|^2\,dH^n \\
&\le C(n, K, R, L)\frac{r^2}{H^n_F(B_r(x))}\int_{B_r(x)}|d f|^2\,dH^n_F.
\end{align*}
This completes the proof.
\end{proof}
\begin{proposition}\label{Poisson}
Let $g \in L^2_{\bfC}(X)$.
Then there exists  $f \in \mathcal{D}^2_{\bfC}(\Delta^F_{\overline{\partial}}, X)$ such that $\Delta^F_{\overline{\partial}}f=g$ holds if and only if 
\begin{align}\label{3}
\int_Xg\,dH^n_F=0.
\end{align} 
Moreover $f$ as above is unique if 
\begin{align}\label{4}
\int_Xf\,dH^n_F=0.
\end{align}
Thus we denote it $(\Delta^F_{\overline{\partial}})^{-1}g$.
\end{proposition}
\begin{proof}
We give a proof of `if' part only because the proof of `only if' part is trivial.
Suppose that (\ref{3}) holds.
Let $\overline{H}^{1, 2}_{\bfC}(X)$ be the closed subspace of $f \in H^{1, 2}_{\bfC}(X)$ with (\ref{4}).
Then by Propositions \ref{L2norm} and \ref{Poincare} we have
\[\int_X|f|^2\,dH^n_F\le C(n, K, d, L)\int_X|\overline{\partial}f|^2\,dH^n_F\]
for every $f \in  \overline{H}^{1, 2}_{\bfC}(X)$.
In particular 
\[||f||_{\overline{H}^{1, 2}_{\bfC}}:=\left(\int_X|d f|^2\,dH^n_F\right)^{1/2}\]
gives a Hilbert norm on $\overline{H}^{1, 2}_{\bfC}(X)$ which is equivalent to $||\cdot ||_{H^{1, 2}_{\bfC}}$.
Let us consider a $\bfC$-linear functional $\mathcal{F}$ on $\overline{H}^{1, 2}_{\bfC}(X)$ defined by
\[\mathcal{F}(\phi):=\int_X\phi\, \overline{g}\,dH^n_F.\]
The Riesz representation theorem yields that there exists a unique $f \in \overline{H}^{1, 2}_{\bfC}(X)$ such that 
\[\mathcal{F}(\phi)=\int_Xh_X^*(\overline{\partial}\phi, \overline{\partial}f)\,dH^n_F\]
for every $\phi \in \overline{H}^{1, 2}_{\bfC}(X)$.
Then it is easy to check that $f \in \mathcal{D}^2_{\bfC}(\Delta^F_{\overline{\partial}}, X)$ with $\Delta^F_{\overline{\partial}}f=g$.
The uniqueness also follows from the argument above.
\end{proof}
\begin{proposition}\label{poissoncont}
Let $g$ be the $L^2$-weak limit on $X$ of a sequence of $g_i \in L^2_{\bfC}(X_i)$ with
\[\int_{X_i}g_i\,dH^n_{F_i}=0.\]
Then $(\Delta^{F_i}_{\overline{\partial}})^{-1}g_i, d((\Delta^{F_i}_{\overline{\partial}})^{-1}g_i)$ $L^2$-converge strongly to $(\Delta^{F}_{\overline{\partial}})^{-1}g, d((\Delta^{F}_{\overline{\partial}})^{-1}g)$ on $X$, respectively.
\end{proposition}
\begin{proof}
Let $f_i:=(\Delta^{F_i}_{\overline{\partial}})^{-1}g_i$.
Propositions \ref{L2norm} and \ref{Poincare} yield
\begin{align*}
\int_{X_i}|f_i|^2\,dH^n_{F_i}&\le C(n, K, d, L)\int_{X_i}|\overline{\partial}f_i|^2\,dH^n_{F_i}\\
&\le C(n, K, d, L)\int_{X_i}g_i\overline{f_i}\,dH^n_{F_i}\\
&\le C(n, K, d, L)\left(\int_{X_i}|g_i|^2\,dH^n_{F_i}\right)^{1/2}\left( \int_{X_i}|f_i|^2\,dH^n_{F_i}\right)^{1/2}. 
\end{align*}
In particular we have $\sup_i||f_i||_{H^{1,2}_{\bfC}}<\infty$.
Thus by Theorem \ref{Rellich} and Proposition \ref{Laplacianweakconv} without loss of generality we can assume that there exists $\hat{f} \in \mathcal{D}^2_{\bfC}(\Delta^F_{\overline{\partial}}, X)$ such that $f_i, d f_i$ $L^2$-converge strongly to $\hat{f},\ d \hat{f}$ on $X$, respectively and that $\Delta^{F_i}_{\overline{\partial}}f_i$ $L^2$-converges weakly to $\Delta^F_{\overline{\partial}}\hat{f}$ on $X$.
Since $\Delta^F_{\overline{\partial}}\hat{f}=g$ and
\[\int_X\hat{f}\,dH^n_F=\lim_{i \to \infty}\int_Xf_i\,dH^n_{F_i}=0,\]
we have $\hat{f}=(\Delta^{F}_{\overline{\partial}})^{-1}g$.
This completes the proof.
\end{proof}
\begin{remark}\label{remark2}
Similar results as above also hold for $\Delta^F$ and $\Delta^F_{\partial}$.
\end{remark}

\begin{proposition}\label{suffsobo}
Assume that $H^n(X \setminus U)=0$ and that
the inclusion
\[H^{1, 2}_{c}(U) \hookrightarrow H^{1, 2}(X)\]
is isomorphic.
Let $f$ be a complex valued function on $U$ such that $f|_O \in H^{1, 2}_{\bfC}(O)$ for every relatively
compact open subset $O$ of $U$ and that $\overline{\partial}f \in L^2((T^*U)'')$ (or $\partial f \in L^2_{\bfC}((T^*U)')$, respectively).
Then we have the following:
\begin{enumerate}
\item[(1)] There exists $u \in H^{1, 2}_{\bfC}(X)$ such that $\overline{\partial}f=\overline{\partial}u$ on $U$ (or $\partial f=\partial u$ on $U$, respectively).
\item[(2)] If $f \in L^1_{\bfC}(U)$, then $f \in H^{1, 2}_{\bfC}(X)$.
\end{enumerate}
In particular, the map
\[H^{1, 2}_{\bfC}(X) \to H^{1, 2}_{\bfC}(U)\]
defined by the restriction is isomorphic.
\end{proposition}
\begin{proof}
We only give a proof of the case of $\overline{\partial}f \in L^2_{\bfC}((T^*U)'')$. 
We first assume $f \in L^{\infty}_{\bfC}(U)$.  By our assumption of the isomorphism 
$H^{1, 2}_{c}(U) \hookrightarrow H^{1, 2}(X)$, 
there exists a sequence $\phi_i \in \mathrm{LIP}_{c, \bfC}(U)$ such that $\phi_i \to 1$ in 
$H^{1, 2}_{\bfC}(X)$.
We use the following notation; 
for any real valued function $g$ and $L_1<L_2$, let
\begin{align*}
g_{L_1}^{L_2}(x):=
\begin{cases} L_2 \,\,\,\,\,\mathrm{if}\,g(x) \ge L_2, \\
g(x) \,\,\,\,\,\mathrm{if}\, L_1<g(x) <L_2, \\
L_1 \,\,\,\,\,\mathrm{if}\, g(x) \le L_1.
\end{cases}
\end{align*}
Since $(\phi_i)_0^1 \in \mathrm{LIP}_{c, \bfC}(U)$ converges to $1$ in $H^{1, 2}_{\bfC}(X)$, 
without loss of generality we can assume that $0 \le \phi_i \le 1$.
Then since $\phi_if \in H^{1, 2}_{\bfC}(X)$, we have
\begin{align*}
&\int_X|d(\phi_i f)|^2\,dH^n\\
&=2\int_X|\overline{\partial}(\phi_if)|^2\,dH^n \\
&=2\int_X\left( |\phi_i|^2|\overline{\partial}f|^2+|f\overline{\partial}\phi_i|^2 + \phi_i\overline{f}h_X^*(\overline{\partial}f, \overline{\partial}\phi_i) + \phi_ifh_X^*(\overline{\partial}\phi_i, \overline{\partial}f) \right)\,dH^n.
\end{align*}
In particular, since $\phi_if \to f$ in $L^2_{\bfC}(X)$ and 
\[\limsup_{i \to \infty}\int_X|d(\phi_i f)|^2\,dH^n<\infty,\]
we have $f \in H^{1, 2}_{\bfC}(X)$.

From now on, we prove Proposition \ref{suffsobo} for general $f$.
For every $L \ge 1$, let $f_L:=(f_1)_{-L}^L+\sqrt{-1}(f_2)^L_{-L}$.
Note that $f_L \in L^{\infty}_{\bf}(X)$, that $f_L|_O \in H^{1, 2}_{\bfC}(O)$ for every 
relatively compact open subset $O$ of $U$ and that $\overline{\partial}f_L \in L^2_{\bfC}((T^*U)'')$ with 
\[\int_X|\overline{\partial}f_L|^2\,dH^n\le \int_X|\overline{\partial}f|^2\,dH^n.\]
Thus from the above, we have $f_L \in H^{1, 2}_{\bfC}(X)$.
By Theorem \ref{Rellich} and Proposition \ref{Poincare}, there exist $u \in H^{1, 2}_{\bfC}(X)$ and a sequence $L_i \to \infty$ such that the functions
\begin{align}\label{functions}
f_{L_i}-\frac{1}{H^n(X)}\int_Xf_{L_i}\,dH^n
\end{align}
converges to $u$ in $L^2_{\bfC}(X)$ and that $\overline{\partial}f_{L_i}$ $L^2$ converges weakly to $\overline{\partial}u$ on $X$.

On the other hand, since $f=f_L$ on $D_L:=\{x \in X; |f(x)|\le L\}$, we have $\overline{\partial}f(x)=\overline{\partial}f_L(x)$ for a.e. $x \in D_L$ (see for instance \cite[Corollary $2.25$]{Cheeger}).
Thus we see that  $\overline{\partial}f_{L_i}$ $L^2$ converges weakly to $\overline{\partial}f$ on $X$.
This completes the proof of (1).

Moreover if $f \in L^1_{\bfC}(X)$, then since the functions (\ref{functions}) converges to
\[f-\frac{1}{H^n(X)}\int_Xf\,dH^n\]
in $L^1_{\bfC}(X)$, we have (2).
\end{proof}


\subsection{The covariant derivative $\nabla''$ in the manner of Gigli}
In this section we define $\nabla''$ for vector fields on nonsmooth setting in the manner of \cite{Gigli}.
For that, let us start giving an observation on smooth setting
(note that in this section we will always consider the nonweighted case).

Let $(M, g_M, J)$ be a compact K\"ahler manifold and let $V$ be a smooth $\bfC$-valued vector field on $M$.
Then it is easy to check that
\begin{eqnarray*}
f_0\,g_M(\nabla '' V, \grad f_1 \otimes df_2)&=& g_M(f_0\,\mathrm{grad}'' f_2, \grad g_M(V, \grad f_1))\\
& & \qquad - f_0\,g_M(V, \nabla_{\mathrm{grad}''f_2}\grad f_1)
\end{eqnarray*}
for any $f_i \in C^{\infty}_{\bfC}(M)$, where $\nabla V=\nabla'V \oplus \nabla''V$ 
with respect to the decomposition $T_{\bfC}X \otimes T^*_{\bfC}X=(T_{\bfC}X \otimes (T^*X)') \oplus (T_{\bfC}X \otimes (T^*X)'')$.
In particular
\begin{eqnarray}\label{6}
&&\int_Mf_0\,g_M(\nabla''V, \grad f_1 \otimes df_2)\,dH^n\\
&=&\int_M\left( - \mathrm{div}(f_0\,\mathrm{grad}''f_2)\,g_M(V, \grad f_1)
 - f_0\,g_M(V, \nabla_{\mathrm{grad}''f_2}\grad f_1 \right) \,dH^n. \nonumber 
\end{eqnarray}
Note that this gives a characterization of $\nabla''V$, that is, if some $T \in L^2_{\bfC}(T_{\bfC}M\otimes T^*_{\bfC}M)$ satisfies
\begin{align*}
&\int_Mf_0\,g_M(T, \grad f_1 \otimes df_2)\,dH^n\nonumber \\
&=\int_M\left( - \mathrm{div}(f_0\mathrm{grad}''f_2)g_M(V, \grad f_1)-f_0\,g_M(V, \nabla_{\mathrm{grad}''f_2}\grad f_1 \right) \,dH^n.
\end{align*}
for any $f_i \in C^{\infty}_{\bfC}(M)$, then $T=\nabla''V$ in $L^2_{\bfC}(T_{\bfC}M\otimes T^*_{\bfC}M)$.
This follows directly from the fact that the space 
\[ \left\{ \sum_{i=1}^N\ f_{0, i}\grad f_{1, i} \otimes df_{2, i}\ ;\  N \in {\bf{N}},\  f_{j, i} \in C^{\infty}_{\bfC}(M)  \right\} \]
is dense in $L^2_{C}(T_{\bfC}M\otimes T^*_{\bfC}M)$.

We will extend this observation to our singular setting.
For this purpose we first give the following definition:
\begin{definition}[Divergence]
Let $\mathcal{D}^2_{\bfC}(\mathrm{div}, U)$ be the set of $V \in L^2_{\bfC}(U)$ such that there 
exists $f \in L^2_{\bfC}(X)$ satisfying
\[\int_Ug_X(V, \grad h)\,dH^n=-\int_Ufh\,dH^n\]
for every $h \in \mathrm{LIP}_{c, \bfC}(U)$.
Since $f$ is unique, we denote it by $\mathrm{div}\,V$.
\end{definition}
\begin{proposition}\label{divconv}
Let $V$ be the $L^2$-weak limit  on $B_R(x)$ of
a sequence $V_i \in \mathcal{D}^2_{\bfC}(\mathrm{div}, B_R(x_i))$ with 
$$\sup_i||\mathrm{div}\,V_i||_{L^2}<\infty.$$
Then we see that $V \in \mathcal{D}^2_{\bfC}(\mathrm{div}, B_R(x))$ and that $\mathrm{div}\,V_i$ $L^2$-converges weakly to $\mathrm{div}\,V$ on $B_R(x)$. 
\end{proposition}
\begin{proof}
By the compactness of $L^2$-weak convergence, without loss of generality we can assume that there exists the $L^2$-weak limit $f$ of $\mathrm{div}\,V_i$ on $B_R(x)$.
Let $h \in \mathrm{LIP}_{c, \bfC}(B_R(x))$.
 By (2.2e), there exists a sequence $h_i \in \mathrm{LIP}_{c, \bfC}(B_R(x))$ such that $h_i,\ d h_i$ $L^2$-converge strongly to $h,\ d h$ on $X$, respectively.
Since
\[\int_{B_R(x_i)}g_{X_i}(V_i, \grad h_i)\,dH^n=-\int_{B_R(x_i)}(\mathrm{div}\,V_i)h_i\,dH^n,\]
we obtain by letting $i \to \infty$
\[\int_{B_R(x)}g_X(V, \grad h)\,dH^n=-\int_{B_R(x)}fh\,dH^n.\]
This gives $V \in \mathcal{D}^2_{\bfC}(\mathrm{div}, B_R(x))$ with $\mathrm{div}\,V=f$.
\end{proof}
\begin{remark}\label{divergenceformula}
It is a direct consequence of simple calculation that for any $V \in \mathcal{D}^2_{\bfC}(\mathrm{div}, U)$ and $f \in \mathrm{LIP}_{\mathrm{loc}, \bfC}(U)$ with $||d f||_{L^{\infty}}<\infty$, we have $fV \in \mathcal{D}^2_{\bfC}(\mathrm{div}, U)$ with
\[\mathrm{div} (fV)=g_X(\grad f, V)+f\mathrm{div}V.\]
\end{remark}
\begin{proposition}\label{derivative}
Let $f \in \mathcal{D}^2_{\bfC}(\Delta, U)$.
Then for every open subset $W$ of $X$ with $\overline{W} \subset U$,  we have $\mathrm{grad}''f|_{W} \in \mathcal{D}^2_{\bfC}(\mathrm{div}, W)$ with  $\mathrm{div}\,(\mathrm{grad}''f)=\mathrm{tr}(\nabla \mathrm{grad}''f)$.
\end{proposition}
\begin{proof}
We first prove the assertion under the assumption $U=X$.
Let $g:=\Delta f$.
By (2.2e), there exists a sequence of $g_i \in C^{\infty}_{\bfC}(X_i)$ with
\[\int_{X_i}g_i\,dH^n=0\]
such that $g_i$ $L^2$-converges strongly to $g$ on $X$.
Let $f_i :=\Delta^{-1}g_i$.
Note that by \cite[Theorem $1.1$]{Honda14b} with Proposition \ref{Laplacianreal} (or Remark \ref{remark2}) we see that $f_i,\ d f_i$ $L^2$-converge strongly to $f,\ d f$ on $X$, respectively and that $\mathrm{Hess}_{f_i}$ $L^2$-converges weakly to $\mathrm{Hess}_f$ on $X$.
Since
\begin{align}\label{9}
\mathrm{div}(\mathrm{grad}''f_i)=\mathrm{tr}(\nabla \mathrm{grad}''f_i)
\end{align}
and 
\begin{eqnarray*}
\nabla \mathrm{grad}''f_i &=&
\nabla \left( \frac{1}{2}\left( \grad f_i + \sqrt{-1}J \grad f_i\right)\right)\\
&=& \frac{1}{2}\nabla \grad f_i+\frac{\sqrt{-1}}{2}\nabla (J\grad f)
\end{eqnarray*}
with $|\nabla J\grad f_i|=|\nabla \grad f_i|$, letting $i \to \infty$ in (\ref{9}) with Proposition \ref{divconv} and \cite[Proposition $3.72$]{Honda13} yields the assertion.

Next we prove the assertion for general $U$.
Since the statement is local, it suffices to check the assertion under $U=B_R(x)$ for some $R>0$ and $x \in X$.
Let $r<R$.
By \cite[Corollary $4.29$]{Honda13}, there exists $\phi \in \mathcal{D}^2_{\bfC}(\Delta, X) \cap \mathrm{LIP}(X)$ such that $0 \le \phi \le 1$, that $\phi|_{B_r(x)}\equiv 1$, that $\mathrm{supp}\,\phi \subset B_R(x)$ and that $\Delta \phi \in L^{\infty}(X)$.
From \cite[Theorem $4.5$]{Honda14b}, we have $\phi f \in \mathcal{D}^2(\Delta, X)$.
Since
\[\mathrm{div}(\mathrm{grad}''(\phi f))=\mathrm{tr}(\nabla \mathrm{grad}''(\phi f)),\]
by restricting this to $B_r(x)$ with Remark \ref{divergenceformula} we have the assertion.
\end{proof}
In order to define $\nabla''$ for vector fields in the manner of \cite{Gigli}, we recall the test class of $\bfR$-valued functions, $\mathrm{Test}F(X)$, defined by Gigli \cite{Gigli} as follows:
\[\mathrm{Test}F(X):=\{f \in \mathcal{D}^2(\Delta, X) \cap \mathrm{LIP}(X); \Delta f \in H^{1, 2}(X)\}.\]
We define the complex version of this as follows:
\[\mathrm{Test}_{\bfC}F(X):=\{f  \in \mathcal{D}^2_{\bfC}(\Delta, X) \cap \mathrm{LIP}_{\bfC}(X); \Delta f \in H^{1, 2}_{\bfC}(X) \}.\]
Proposition \ref{Laplacianreal} yields that for every $\bfC$-valued function $f$ on $X$, $f \in \mathrm{Test}_{\bfC}F(X)$ holds if and only if $f^i \in \mathrm{Test}F(X)$ holds for every $i=1, 2$, where $f=f^1+\sqrt{-1}f^2$ and $f^i$ is $\bfR$-valued.
On the other hand it is known in \cite{Gigli, Honda14b} that the space
\[\mathrm{Test}(T^1_1X):=\left\{\sum_{i=1}^Nf_{0, i}\grad f_{1, i}\otimes df_{2, i}; N \in {\bf{N}}, f_{j, i} \in \mathrm{Test}FX\right\}\]
is dense in $L^2(TX \otimes T^*X)$.
This gives that the space
\[\mathrm{Test}_{\bfC}((T^1_1)_{\bfC}X):=\left\{\sum_{i=1}^Nf_{0, i}\grad f_{1, i}\otimes df_{2, i}; N \in {\bf{N}}, f_{j, i} \in \mathrm{Test}_{\bfC}FX\right\}\]
is also dense in $L^2_{\bfC}(T_{\bfC}X \otimes T^*_{\bfC}X)$.
\begin{definition}[$\nabla''$ for vector fields in the manner of Gigli]\label{nabladef}
Let $\mathcal{D}^2_{\bfC}(\nabla'', X)$ be the set of $V \in L^2_{\bfC}(T_{\bfC}X)$ such that 
there exists $T \in L^2_{\bfC}(TX \otimes T^*_{\bfC}X)$ satisfying
\begin{align}\label{10}
&\int_Xf_0\,g_X(T, \grad f_1 \otimes df_2)\,dH^n  \\
&=\int_X\left( - \mathrm{div}(f_0\,\mathrm{grad}''f_2)\,g_X(V, \grad f_1) - f_0\,g_X(V, \nabla_{\mathrm{grad}''f_2}\grad f_1\right)\,dH^n \nonumber
\end{align}
for any $f_i \in \mathrm{Test}_{\bfC}F(X)$.
Since $T$ is unique, we denote it by $\nabla''V$.
\end{definition}
The following stability result for $\nabla''$ with respect to the Gromov-Hausdorff topology 
plays a key role in this paper:
\begin{proposition}\label{stability}
Let $V$ be the $L^2$-strong limit on $X$ of a sequence of $V_i \in \mathcal{D}^2_{\bfC}(\nabla'', X_i)$ with $\sup_i||\nabla''V_i||_{L^2}<\infty$. 
Then we see that $V \in \mathcal{D}^2_{\bfC}(\nabla'', X)$ and that $\nabla''V_i$ $L^2$-converges weakly to $\nabla''V$ on $X$.
\end{proposition} 
\begin{proof}
By the compactness of $L^2$-weak convergence, without loss of generality we can assume that there exists the $L^2$-weak limit $T$ of $\nabla''V_i$ on $X$.

We first prove:
\begin{claim}\label{11}
The equation (\ref{10}) holds if $\Delta f_i \in \mathrm{LIP}_{\bfC}(X)$ holds for $i = 1,\ 2$.
\end{claim}
\noindent
The proof is as follows.
Suppose that $\Delta f_i \in \mathrm{LIP}_{\bfC}(X)$ holds for $i = 1,\ 2$.
Let $g_i := \Delta f_i$.
By (2.2e) there exists a sequence of $g_{i, j} \in \mathrm{LIP}_{\bfC}(X_j)$ such that $\sup_{i, j}||d g_{i, j}||_{L^{\infty}}<\infty$, that 
\[ \int_{X_j}g_{i, j}\,dH^n=0\]
and that $g_{i, j},\ d g_{i, j}$ $L^2$-converge strongly to $g_i,\ d g_i$ on $X$, respectively.
Let $f_{i, j}:=\Delta^{-1}g_{i, j}$.
By Proposition \ref{Laplacianreal}, Remark \ref{remark2}, \cite[Theorems $1.1$ and $4.13$]{Honda14b}, we see that $f_{i, j} \in \mathrm{LIP}_{\bfC}(X_j)$, that $\sup_{i, j}||f_{i, j}||_{L^{\infty}}<\infty$, that $f_{i, j},\ d f_{i, j}$ $L^2$-converge strongly to $f_i,\ d f_i$ on $X$, respectively, and that $\mathrm{Hess}_{f_{i, j}}$ $L^2$-converges weakly to $\mathrm{Hess}_{f_i}$ on $X$.
In particular, $f_{i, j} \in \mathrm{Test}_{\bfC}F(X)$ and $df_{2, j}$ $L^2$-converges strongly to $d f_2$ on $X$.
Since
\begin{eqnarray*}
&&\int_{X_j}f_{0, j}g_{X_j}(\nabla''V_i, \grad f_{1, j}\otimes df_{2, j})\,dH^n \\
&& \qquad\qquad =\int_{X_j}\left(- \mathrm{div}(f_{0, j}\mathrm{grad}'' f_{2, j})g_{X_j}(V_j, \grad f_{1, j})\right.\\
&& \qquad\qquad\qquad\qquad \left.-f_{0, j}g_{X_j}(V_j, \nabla_{\mathrm{grad}'' f_{2, j}}\grad f_{1, j})\right)\,dH^n,
\end{eqnarray*}
by letting $j \to \infty$ with Proposition \ref{divconv} we have Claim \ref{11}.

The following is shown in \cite[Proposition 7.5]{Honda14b}.
For reader's convenience we give the proof:
\begin{claim}\label{12}
Let $g \in \mathrm{Test}_{\bfC}F(X)$.
Then there exists a sequence $g_k \in \mathrm{Test}_{\bfC}F(X)$ with $\Delta g_k \in \mathrm{LIP}_{\bfC}(X)$ and $\sup_k||d g_k||_{L^{\infty}}<\infty$ such that $g_k, \Delta g_k \to g, \Delta g$ in $H^{1, 2}_{\bfC}(X)$, respectively.
\end{claim}
\noindent
The proof is as follows.
Let
\[h_{\delta, \epsilon}g_k:=h_{\delta}(\widetilde{h}_tg^1_k)+\sqrt{-1}h_{\delta}(\widetilde{h}_tg^2_k),\]
where $g_k=g_k^1+\sqrt{-1}g_k^2$, $h_t$ is the heat flow on $X$ and $\widetilde{h}_t$ is a mollified heat flow defined by
\[\widetilde{h}_tg_k:=\frac{1}{t}\int_0^{\infty}h_sg_k\phi (st^{-1})ds\]
for some nonnegatively valued smooth function $\phi$ on $(0, 1)$ with
\[\int_0^1\phi ds=1\]
(see for instance \cite{AGS} for the heat flow and \cite[(3.2.3)]{Gigli} for a mollified heat flow).
From the regularity of the heat flow \cite{AGS, Gigli} with Proposition \ref{Laplacianreal} we have the following:
\begin{enumerate}
\item $h_{\delta, \epsilon}g \in \mathrm{Test}_{\bfC}F(X)$.
\item $\Delta h_{\delta, \epsilon}g \in \mathrm{LIP}_{\bfC}(X)$.
\item $\sup_{\delta, \epsilon<1}||\nabla (h_{\delta, \epsilon}g)||_{L^{\infty}}<\infty$.
\item $h_{\delta, \epsilon}g, \Delta h_{\delta, \epsilon}g  \to \widetilde{h}_{\epsilon}g, \Delta \widetilde{h}_{\epsilon}g$ in $H^{1, 2}_{\bfC}(X)$, respectively as $\delta \to 0$.
\item $\widetilde{h}_{\epsilon}g, \Delta \widetilde{h}_{\epsilon}g \to g, \Delta g$ in $H^{1, 2}_{\bfC}(X)$, respectively as $\epsilon \to 0$.
\end{enumerate}
This completes the proof of Claim \ref{12}.

We are now in a position to finish the proof of Proposition \ref{stability}.
Let $f_i \in \mathrm{Test}_{\bfC}F(X)$.
Then Claim \ref{12} yields that there exists a sequence $f_{i, j} \in \mathrm{Test}_{\bfC}F(X)$ such that $\Delta f_{i, j} \in \mathrm{LIP}_{\bfC}(X)$, that $\sup_{i, j}||\nabla f_{i, j}||_{L^{\infty}}<\infty$, and that $f_{i, j}, \Delta f_{i, j} \to f_i, \Delta f_i$ in $H^{1, 2}_{\bfC}(X)$.  
Note that by \cite[Theorem $1.2$]{Honda13}, $\mathrm{Hess}_{f_{i, j}}$ $L^2$-converges weakly to $\mathrm{Hess}_{f_i}$ on $X$.
Claim \ref{11} yields
\begin{align*}
&\int_Xf_{0, j}g_X(T, \grad f_{1, j} \otimes df_{2, j})\,dH^n\\
&=\int_X\left(- \mathrm{div}(f_{0, j}\mathrm{grad}''f_{2, j})g_X(V, \grad f_{1, j})-f_{0, j}g_X(V, \nabla_{\mathrm{grad}''f_{2, j}}\grad f_{1, j}\right)\,dH^n.
\end{align*}
By letting $j \to \infty$ we have $V \in \mathcal{D}^2_{\bfC}(\nabla'', X)$ with $\nabla''V=T$.
This completes the proof.    
\end{proof}
We end this section by giving a compatibility between our setting and smooth setting:
\begin{proposition}\label{compatibility}
Suppose that $(U, g_X|_U, J|_U)$ is a smooth K\"ahler manifold.  
Then for every $V \in C^{\infty}_{\bfC}(U)$, $\nabla''V$ in the sense of Definition \ref{nabladef} coincides with the ordinary one.
\end{proposition}
\begin{proof}
Let $S:=\nabla'' V$ be as in the sense of Definition \ref{nabladef} and let $T:=\nabla'' V$ be as in the ordinary sense.
From (\ref{6}) and Definition \ref{nabladef} we have
\begin{align*}
\int_Ug_X(S, f_0\grad f_1 \otimes df_2)\,dH^n=\int_Ug_X(T, f_0\grad f_1 \otimes df_2)\,dH^n
\end{align*}
for any $f_i \in C^{\infty}_{c, \bfC}(U)$.
Since the space 
\[\left\{ \sum_{i=1}^Nf_{0, i}\grad f_{1, i} \otimes df_{2, i}; N \in \mathbf{N}, f_{j, i} \in C^{\infty}_{c, \bfC}(U)\right\}\]
is dense in $L^2_{\bfC}(T_{\bfC}U \otimes T^*_{\bfC}U)$, we have $S=T$.
\end{proof}

\section{Fano-Ricci limit spaces}
\subsection{Definition of Fano-Ricci limit spaces}
In this subsection, besides (2.1a) - (2.1e), we add the following assumptions:
\begin{enumerate}
\item[(4.1a)] $X_i$ is an $m$-dimensional 
Fano manifold with the K\"ahler form $\omega_i$ in $2\pi c_1(X_i)$ for every $i$.
\item[(4.1b)]  For every $i$, $F_i$ is the Ricci potential, i.e.
\[\mathrm{Ric}(\omega_i)-\omega_i= \sqrt{-1}\partial \overline{\partial}F_i.\]
with the normalization
$$ \int_{X_i} e^{F_i} \omega_i^m = \int_{X_i} \omega^m,$$
or equivalently $H^n_{F_i}(X_i)=H^n(X_i)$. (Recall $n = 2m$.)
\end{enumerate}
Then we call $(X, g_X, J, F)$ the \textit{Fano-Ricci limit space of $(X_i, g_{X_i}, J_i, F_i)$} or the \textit{Fano-Ricci limit space} for short. Since $F_i$ is uniquely determined by $g_{X_i}$ we shall omit $F$ and $F_i$,
and write $(X, g_X, J)$ and $(X_i, g_{X_i}, J_i)$ if no confusion is likely to occur.
\begin{theorem}[Weitzenb\"ock inequality]\label{weith}
Let $f \in \mathcal{D}^2_{\bfC}(\Delta^F_{\overline{\partial}}, X)$.
Then we have $\mathrm{grad}'f \in \mathcal{D}^2_{\bfC}(\nabla'', X)$ with $\nabla'' \mathrm{grad}'f \in L^2_{\bfC}(T_{\bfC}'X \otimes (T^*_{\bfC}X)'')$ and
\begin{align}\label{15}
\int_X|\Delta^F_{\overline{\partial}}f|^2\,dH^n_F \ge \int_X|\nabla '' \mathrm{grad}'f|^2\,dH^n_F+\int_X|\overline{\partial}f|^2\,dH^n_F.
\end{align}
\end{theorem}
\begin{proof}
Let $g:=\Delta^F_{\overline{\partial}}f$.
By (2.2e), 
there exists a sequence of $g_i \in C^{\infty}_{\bfC}(X_i)$ such that 
$g_i,\ dg_i$ $L^2$-converge strongly to $g,\ dg$ on $X$, respectively and that 
\[\int_{X_i}g_i\,dH^n_{F_i}=0.\]
Let $f_i:=(\Delta^{F_i}_{\overline{\partial}})^{-1}g_i$. 
Propositions \ref{convergencedecomp} and \ref{poissoncont} yield that $\overline{\partial}f_i$ $L^2$-converges strongly to $\overline{\partial}f$ on $X$.
Now we use the Weitzenb\"ock formula on a Fano manifold (see 
\cite[page 41]{Futaki88})
\begin{align}\label{16}
\int_{X_i}|\Delta^{F_i}_{\overline{\partial}}f_i|^2\,dH^n_{F_i} = \int_{X_i}|\nabla '' \mathrm{grad}'f_i|^2\,dH^n_{F_i}+\int_{X_i}|\overline{\partial}f_i|^2\,dH^n_{F_i}.
\end{align}
In particular we have $\sup_i||\nabla'' \mathrm{grad}'f_i||_{L^2}<\infty$.
Thus by Remark \ref{decompremark} and Proposition \ref{stability}, we see that $\mathrm{grad}'f \in \mathcal{D}^2_{\bfC}(\nabla'', X)$, that $\nabla'' \mathrm{grad}'f \in L^2_{\bfC}(T_{\bfC}'X \otimes T''X)$ and that
$\nabla'' \mathrm{grad}'f_i$ $L^2$-converges weakly to $\nabla''\mathrm{grad}'f$ on $X$.
Thus by taking $i \to \infty$ in (\ref{16}) we have (\ref{15}).
\end{proof}
\begin{corollary}\label{1steigen}
We have the following.
\begin{enumerate}
\item[(1)] $\lambda_1(\Delta_{\overline{\partial}}^F, X) \ge 1$.
\item[(2)] If $f \in \mathcal{D}^2_{\bfC}(\Delta_{\overline{\partial}}^F, X)$ with $\Delta^F_{\overline{\partial}}f=f$, then
$\nabla''\mathrm{grad}'f=0$. In particular if $(U, g_X|_U, J|_U)$ is a smooth K\"ahler manifold with $F|_U \in C^{\infty}(U)$, then $f|_U \in C^{\infty}_{\bfC}(U)$ and $\mathrm{grad}'f$ is a holomorphic vector field on $U$. 
\end{enumerate}
\end{corollary}
\begin{proof}
Let $f \in \mathcal{D}^2_{\bfC}((\Delta_{\overline{\partial}}^F, X)$ be a $\lambda$-eigenfunction of $\Delta^F_{\overline{\partial}}$ on $X$.
Then Theorem \ref{weith} yields
\[(\lambda-1)\int_X|f|^2\,dH^n_F\ge \int_X|\nabla'' \mathrm{grad}'f|^2\,dH^n_F.\]
This proves (1). This also shows that if $f \in \mathcal{D}^2_{\bfC}((\Delta_{\overline{\partial}}^F, X)$ with $\Delta^F_{\overline{\partial}}f=f$ then
$$\nabla''\mathrm{grad}'f=0.$$
Finally we assume that $(U, g_X|_U, J)$ is a smooth K\"ahler manifold with $F|_U \in C^{\infty}(U)$.
Then Proposition \ref{Laplacianexplicite} and the elliptic regularity theorem yield $f|_U \in C^{\infty}_{\bfC}(U)$.
Thus Proposition \ref{compatibility} yields that $\mathrm{grad}'f$ is a holomorphic vector field on $U$. 
\end{proof}
\begin{remark}\label{eigenhol}
Corollary \ref{1steigen} with Proposition \ref{L2norm} gives that a $\bfC$-linear map 
\begin{eqnarray*}
\Phi: \Lambda_1=\Lambda_1(X)&:=&\left\{f \in \mathcal{D}^2_{\bfC}(\Delta^F_{\overline{\partial}}, X); \Delta^F_{\overline{\partial}}f=f\right\}\\
&& \qquad \to L^2_{\bfC}(T'X) \cap \left\{V \in \mathcal{D}^2_{\bfC}(\nabla'', X); \nabla''V=0\right\}
\end{eqnarray*}
defined by $\Phi(f):=\mathrm{grad}'f$ is injective.
\end{remark}

Let $\mathfrak{h}_1(X)$ be the set of $V \in L^2_{\bfC}(T'X)$ with $V=\mathrm{grad}'u$ for some $u \in \Lambda_1$ (i.e. $\mathfrak{h}_1(X)=\Phi(\Lambda_1)$).
It is known (\cite{Futaki87}, \cite{Futaki88}) that, if $(X, g_X, J)$ is a smooth Fano manifold, then $\mathfrak{h}_1(X)$ coincides with the space of all holomorphic vector fields on $X$.
\begin{proposition}\label{uppersemi}
Let $V_i \in \mathfrak{h}_1(X_i)$ be a sequence with $\sup_i||V_i||_{L^2}<\infty$.
Then there exist a subsequence $\{i(j)\}_j$ and $V \in \mathfrak{h}_1(X)$ such that $V_{i(j)}$ $L^2$-converges strongly to $V$ on $X$.
In particular,
\[\limsup_{i \to \infty}\mathrm{dim}\,\mathfrak{h}_1(X_i)\le \mathrm{dim}\,\mathfrak{h}_1(X)<\infty.\]
\end{proposition}
\begin{proof}
Let $u_i \in \Lambda_1(X_i)$ with $V_i=\mathrm{grad}'u_i$.
By the proof of Proposition \ref{poissoncont}, we have $\sup_i||u_i||_{H^{1, 2}_{\bfC}}<\infty$.
Thus by Theorem \ref{Rellich}, without loss of generality we can assume that there exists the $L^2$-strong limit $u \in H^{1, 2}_{\bfC}(X)$ of $u_i$ on $X$.
Proposition \ref{poissoncont} gives that $u_i, du_i$ $L^2$-converge strongly to $u, du$ on $X$, respectively.
In particular, by Propositions \ref{convergencedecomp} and \ref{Laplacianweakconv}, we see that $u \in \Lambda_1$ and that $\mathrm{grad}'u_i$ $L^2$-converges strongly to $\mathrm{grad}'u \in \mathfrak{h}_1(X)$ on $X$.

Note that the finite dimensionality of $\mathfrak{h}_1(X)$ follows from that of $\Lambda_1$.
Thus this completes the proof.
\end{proof}

\begin{remark}\label{extend}
It is easy to check that the similar results as above hold even if each $(X_i, g_{X_i}, J_i, F_i),$ is a 
Fano-Ricci limit spaces.
\end{remark}

Finally we define the Futaki invariant of the Fano-Ricci limit space $(X, g_X, J, F)$ as a $\bfC$-valued linear function on $\mathfrak{h}_1(X)$:
\begin{definition}
We define $\mathcal{F}: \mathfrak{h}_1(X) \to \bfC$ by
\[\mathcal{F}_X(V, g_X):=\int_XV(F)\,dH^n.\]
\end{definition}

\begin{proposition}\label{Futaki1}
We have the following:
\begin{enumerate}
\item[(1)] Let $V \in \mathfrak{h}_1(X)$ with $V=\mathrm{grad}'u$ for some $u \in \Lambda_1$. Then
\[\mathcal{F}_X(V,g_X):=-\int_Xu\,dH^n.\]
\item[(2)] Let $V_i$ be a sequence in $\mathfrak{h}_1(X_i)$ and let $V \in \mathfrak{h}_1(X)$ be the $L^2$-strong limit on $X$.
Then 
\[\lim_{i \to \infty}\mathcal{F}_{X_i}(V_i, g_{X_i})=\mathcal{F}_X(V,g_X).\]
\end{enumerate}
\end{proposition}
\begin{proof}
We first prove (1).
By Remark \ref{L1lap}, we have $u \in \mathcal{D}^{2, 1}_{\bfC}(\Delta_{\barpartial}, X)$ with 
\[u=\Delta^F_{\barpartial}u=\Delta_{\barpartial}u-h_X^*(\barpartial u, \barpartial F).\]
Integrating this with respect to $dH^n$ yields (1).

(2) is a direct consequence of the proof of Proposition \ref{uppersemi} and (1).
\end{proof}

We say two Fano-Ricci limit spaces $(X, g_X, J_X)$ and $(Y, g_Y, J_Y)$ are $J$-equivalent if
the following condition holds: If
$(X_i, g_{X_i}, J_{X_i})$ and $(Y_i, g_{Y_i}, J_{Y_i})$ converge to $(X, g_X, J_X)$ and $(Y, g_Y, J_Y)$ 
then there are biholomorphic automorphisms $\psi_i$ of $(X_i, J_{X_i})$ to $(Y_i, J_{Y_i})$ for all $i$.
Further, we say that 
$V \in \mathfrak{h}_1(X)$ and $W \in \mathfrak{h}_1(Y)$ are $J$-equivalent if the following condition holds:
if $V \in \mathfrak{h}_1(X)$ is an $L^2$-strong limit of 
a sequence in $V_i \in \mathfrak{h}_1(X_i)$ with respect to $g_{X_i}$ 
then $W \in \mathfrak{h}_1(Y)$ is an $L^2$-strong limit of  $(\psi_i)_\ast V_i$ 
with respect to $g_{Y_i}$.

\begin{theorem}\label{Futaki2} If two
Fano-Ricci limit spaces $(X, g_X, J_X)$ and $(Y, g_Y, J_Y)$ are $J$-equivalent and if
two vector fields $V \in \mathfrak{h}_1(X)$ and $W \in \mathfrak{h}_1(Y)$ are $J$-equivalent
then 
$$F_X(V,g_X) = F_Y(W,g_Y).$$
\end{theorem}

\begin{proof} It is trivial to have
$$
 \mathcal{F}_{X_i}(V_i, g_{X_i})=\mathcal{F}_{Y_i}((\psi_i)_\ast V_i,(\psi^{-1}_i)^\ast g_{X_i}).
 $$
But since $\mathcal{F}_{Y_i}((\psi_i)_\ast V_i,g_{Y_i})$ is independent of the choice of
the K\"ahler metric $g_{Y_i}$ with the K\"ahler metric in the anti-canonical class by \cite{futaki83.1}, 
we have 
$$
\mathcal{F}_{Y_i}((\psi_i)_\ast V_i,(\psi^{-1}_i)^\ast g_{X_i})
=\mathcal{F}_{Y_i}((\psi_i)_\ast V_i,g_{Y_i}).
 $$
Taking the limit as $i \to \infty$ and using Proposition \ref{Futaki1}, (2), we obtain
$F_X(V,g_X) = F_Y(W,g_Y)$. This completes the proof.
\end{proof}

\begin{proposition}\label{Futaki3}
We have the following:
\begin{enumerate}
\item[(1)] We have $||\mathcal{F}_X|||_{\mathrm{op}} \le C(n, K, d, L)$, where  $||\mathcal{F}_X||_{\mathrm{op}}$ is the operator norm of $\mathcal{F}_X$, i.e.
\[||\mathcal{F}_X||_{\mathrm{op}}:=\sup_{||V||_{L^2}=1}|\mathcal{F}_X(V)|.\]
\item[(2)] We have
\[\limsup_{i \to \infty}||\mathcal{F}_{X_i}||_{\mathrm{op}} \le ||\mathcal{F}_X||_{\mathrm{op}}.\]
\item[(3)] If 
\[\lim_{i \to \infty}\mathrm{dim}\,\mathfrak{h}_1(X_i)=\mathrm{dim}\,\mathfrak{h}_1(X),\]
then
\[\lim_{i \to \infty}||\mathcal{F}_{X_i}||_{\mathrm{op}} = ||\mathcal{F}_X||_{\mathrm{op}}.\]
\end{enumerate}
\end{proposition}
\begin{proof}
We first prove (1).
Let $u \in \Lambda_1$ and let $V=\mathrm{grad}'u$.
Then from Proposition \ref{Poincare} and (1) of Proposition \ref{Futaki1}, we have
\begin{align*}
|\mathcal{F}_X(V)| &\le \int_X|u|dH^n\\
&\le C(L)\int_X|u|dH^n_F \\
&\le C(n, K, d, L) \int_X|\mathrm{grad}'u|^2dH^n_F \\
&\le C(n, K, d, L) \int_X|V|^2dH^n \le C(n, K, d, L).
\end{align*} 
This completes the proof of (1).

Next we prove (2).
For every $i<\infty$, there exists $V_i \in \mathfrak{h}_1(X_i)$ such that $||V_i||_{L^2}=1$ and that $|\mathcal{F}_{X_i}(V_i)|=||\mathcal{F}_{X_i}||_{\mathrm{op}}$ holds because $\mathfrak{h}_1(X_i)$ is finite dimensional.
By Proposition \ref{uppersemi}, without loss of generality we can assume that there exists $V \in \mathfrak{h}_1(X)$ such that $V_i$ $L^2$-converges strongly to $V$ on $X$. 
Thus (2) of Proposition \ref{Futaki1} yields
\[\limsup_{i \to \infty}||\mathcal{F}_{X_i}||_{\mathrm{op}} =\limsup_{i \to \infty}|\mathcal{F}_{X_i}(V_i)| =|\mathcal{F}_X(V)| \le ||\mathcal{F}_X||_{\mathrm{op}}.\]

Finally we prove (3).
Let $V \in \mathfrak{h}_1(X)$ with $|\mathcal{F}_X(V)|=||\mathcal{F}_X||_{\mathrm{op}}$ and let $\{i(j)\}_j$ be a subsequence.
Then by Proposition \ref{uppersemi}, there exist a subsequence $\{j(k)\}_j$ of $\{i(j)\}_j$  and  a sequence $V_{j(k)} \in \mathfrak{h}_1(X_{j(k)})$ such that $||V_{j(k)}||_{L^2}=1$ and that $V_{j(k)}$ $L^2$-converges strongly to $V$ on $X$.
Thus applying (2) of Proposition \ref{Futaki1} again yields
\[||\mathcal{F}_X||_{\mathrm{op}}=|\mathcal{F}_X(V)| = \lim_{k \to \infty}|\mathcal{F}_{X_{j(k)}}(V_{j(k)})| \le \liminf_{k \to \infty}||\mathcal{F}_{X_{j(k)}}||_{\mathrm{op}}.\]
Since $\{i(j)\}_j$ is arbitrary, this completes the proof of (3).

\end{proof}

\subsection{A compactness with respect to the Gromov-Hausdorff convergence}
In this subsection we start without assuming (2.1d), and rather study when (2.1d) is satisfied, 
see Proposition \ref{lowerbound}.
\begin{proposition}\label{upperbound}
Let $M$ be a Fano manifold with $\mathrm{Ric}_M \ge K$ and $\mathrm{diam}\,M \le d$, and let $F$ be the Ricci potential with the canonical normalization.
Then we have
\begin{align}\label{13}
F\le C(n, K, d)
\end{align}
and 
\begin{align}\label{14}
\frac{1}{H^n(M)}\int_M|\nabla e^F|^2\,dH^n\le C(n, K, d).
\end{align}
\end{proposition}
\begin{proof}
By taking the (complex) trace of the equation $\mathrm{Ric}(\omega)-\omega=i\partial \overline{\partial} F$ we have
$s_M/2-n=-\Delta F/2$, where $s_M$ is the scalar curvature of $M$ in the sense of Riemannian geometry.
Then since
\[\Delta e^F=-e^F|\nabla F|^2 + e^F\Delta F \le e^F(2n-s_M) \le (2-K)n e^F,\]
Li-Tam's mean value inequality \cite[Corollary $3.6$]{Li} (or \cite[Theorem $1.1$]{LT}) yields 
\[e^F \le C(n, K, d)\frac{1}{H^n(M)}\int_Me^F\,dH^n=C(n, K, d).\]
Thus we have (\ref{13}).

On the other hand, since
\[\Delta e^{2F}=-4e^{2F}|\nabla F|^2+2e^{2F}\Delta F=-4e^{2F}|\nabla F|^2+2e^{2F}(2n-s_M),\]
by integration of this on $M$, we have
\[2\int_Me^{2F}|\nabla F|^2\,dH^n \le \int_Me^{2F}(2n-s_M)\,dH^n \le (2-K)n\int_Me^{2F}\,dH^n\le C(n, K, d).\]
This gives (\ref{14}).
\end{proof}
The following is a direct consequence of \cite[Theorem $4.9$]{Honda13}, \cite[Theorem $6.19$]{Honda14b} and Proposition \ref{upperbound}:
\begin{corollary}\label{compactness}
Let $K \in \bfR$, let $d, v>0$ and let $n \in \mathbf{N}$
Let $X_i$ be a sequence of Fano manifolds with $\mathrm{Ric}_{X_i}\ge K$, $\mathrm{diam}\, X_i \le d$, and $H^n(X_i) \ge v$.


Then there exist a subsequence $X_{i(j)}$, the noncollapsed Gromov-Hausdorff limit $X$, the $L^2$-strong limit $J$ of $J_{i(j)}$ on $X$, and the $L^2$-strong limit $G \in H^{1, 2}(X) \cap L^{\infty}(X)$ of $e^{F_{i(j)}}$ on $X$ such that $\nabla e^{F_{i(j)}}$ $L^2$-converges weakly to $\nabla G$ on $X$, where $F_i$ is the Ricci potential of $X_i$ with the canonical normalization.
Moreover, if there exists $c \in \bf{R}$ such that $F_i \ge c$ for every $i<\infty$, then there exists $F \in L^{\infty}(X)$ with $c \le F \le C(n, K, d)$ such that $G=e^F$.
In particular $F_i$ $L^2$-converges strongly to $F$ on $X$.
\end{corollary}
\begin{corollary}\label{unfbound}
Let $(X, g_X, J, F)$ be a Fano-Ricci limit space with $H^n(X) \ge v$, $F \ge c$ and $\mathrm{diam}\,X \le d$. 
Then
\[0<C_1(n, K, d, v, c, l)\le \lambda_l(\Delta^F_{\overline{\partial}}, X)\le C_2(n, K, d, v, c, l)<\infty\]
for every $l \ge 1$.
\end{corollary}
\begin{proof}
We only give a proof of the existence of upper bounds because the proof of the existence of lower bounds is similar.
The proof is done by a standard contradiction.
Assume that the assertion is false.
Then there exist $l \ge 1$ and a sequence of Fano-Ricci limit spaces $(X_i, g_{X_i}, J_i, F_i)$ with $H^n(X_i)\ge v$, $\mathrm{diam}\,X_i \le d$, $F_i \ge c$ and
\[\lim_{i \to \infty}\lambda_l(\Delta^{F_i}_{\overline{\partial}}, X_i) = \infty.\]
On the other hand by Corollary \ref{compactness} we can assume without loss of generality 
that there exist the noncollapsed Gromov-Hausdorff limit $X$ 
of $X_i$, the $L^2$-strong limit $J$ of $J_{i}$ on $X$, and 
the $L^2$-strong limit $F \in H^{1, 2}(X) \cap L^{\infty}(X)$ of $F_{i}$ on $X$.
Since Proposition \ref{spectralconv} yields 
\[\lim_{i \to \infty}\lambda_l(\Delta^{F_i}_{\overline{\partial}}, X_i)=\lambda_l(\Delta^F_{\overline{\partial}}, X)<\infty,\]
this is a contradiction.
\end{proof}

Similarly, we have the following:
\begin{corollary}\label{dimensionbound}
Under the same assumption as in Corollary \ref{unfbound},  we have
\[\mathrm{dim}\,\mathfrak{h}_1(X)\le C(n, K, d, v, c).\]
\end{corollary}
\begin{proof}
The proof is done by a contradiction.
Assume that the assertion is false.
Then there exist a sequence of Fano-Ricci limit spaces $(X_i, J_i, g_i, F_i)$ with $\mathrm{Ric}_{X_i}\ge K$, $H^n(X_i)\ge v$, $\mathrm{diam}\,X_i \le d$, $F_i \ge c$ and
\[\lim_{i \to \infty}\mathrm{dim}\,\mathfrak{h}_1(X_i) = \infty.\]
Let $\{V_{j, i}\}_{j=1}^{\mathrm{dim}\,\mathfrak{h}_1(X_i)}$ be an $L^2$-orthogonal basis of $\mathfrak{h}_1(X_i)$.
By Proposition \ref{uppersemi} and Corollary \ref{compactness}, we can assume 
without loss of generality that there exist the noncollapsed Gromov-Hausdorff limit $X$ of $X_i$, 
the $L^2$-strong limit $J$ of $J_{i}$ on $X$,  
the $L^2$-strong limit $F \in H^{1, 2}(X) \cap L^{\infty}(X)$ of $F_{i}$ on $X$, 
and the $L^2$-strong limits $V_i \in \mathfrak{h}_1(X)$ of $V_{j, i}$ on $X$.
This contradicts the finite dimensionality of $\mathfrak{h}_1(X)$.
\end{proof}

We give a sufficient condition to get a uniform lower bound on the Ricci potential by a standard way of Riemannian geometry:
\begin{proposition}\label{lowerbound}
Let $q>n/2$, let $\hat{L}>0$, and
let $M$ be a Fano manifold with $\mathrm{Ric}_M \ge K$, $\mathrm{diam}\,M \le d$, $H^n(X) \ge v$ and
\[\int_M|s_M|^q\,dH^n \le \hat{L}.\]
Then the Ricci potential $F$ of $M$ with the canonical normalization satisfies
\[|F| \le C(n, K, d, v, q, \hat{L}).\]
\end{proposition}
\begin{remark}\label{holderlipschitz}
In Proposition \ref{lowerbound} if $n/2<q<n$, by \cite[Theorem $1.2$]{Jiang11} with \cite[Theorem $5.1$]{HK} we have the following quantitative H\"older continuity of $F$:
\[|F(x)-F(y)|\le C d(x, y)^{\alpha}\]
for any $x, y \in M$, where $C:=C(n, K, d, v, q, \hat{L})>0$ and $\alpha = 2-n/q$.
Moreover if $q>n$, then we have the quantitative Lipschitz continuity of $F$:
\[|\mathrm{grad}\,F|\le C(n, K, d, v, q, \hat{L}).\] 
See \cite[Theorem $1.2$]{Jiang12}.
\end{remark}
In summary, we have the following compactness.
\begin{corollary}
Let $q>n/2$, let $\hat{L}>0$ and let $X_i$ be a sequence of Fano manifolds with $\mathrm{Ric}_{X_i}\ge K$, $\mathrm{diam}\, X_i \le d$, $H^n(X_i) \ge v$ and
\[\int_{X_i}|s_{X_i}|^q\,dH^n \le \hat{L}.\]
Then there exist a subsequence $M_{i(j)}$, the noncollapsed Gromov-Hausdorff limit $M$, the $L^2$-strong limit $J$ of $J_{i(j)}$ on $M$, and the $L^2$-strong limit $F \in H^{1, 2}(M) \cap L^{\infty}(M)$ of $F_{i(j)}$ on $M$ such that $\sup_j||F_{i(j)}||_{L^{\infty}} \le C(n, K, d, v, q, \hat{L})$ and that $\nabla {F_{i(j)}}$ $L^2$-converges weakly to $\nabla F$ on $M$.
\end{corollary}

\subsection{The Lie algebra structure of subspaces of $\Lambda_1$ on nonsmooth setting}
In this section we discuss a subspace of $\Lambda_1$ which is a Lie algebra by the Poisson bracket $\{ \cdot, \cdot \}$.
\begin{definition}[Poisson bracket]\label{poissondefinition}
Let $(X, g_X, J)$ be the noncollapsed K\"ahler-Ricci limit space of $(X_i, g_{X_i}, J_i)$, i.e. (2.1a)-(2.1c) holds and $J$ is the $L^2$-strong limit of $J_i$ on $X$.
Then for any open subset $U$ of $X$, and $u, v \in H^{1, 2}_{\bfC}(U)$, let
\[\{u, v\}:=\mathrm{grad}'u(v)-\mathrm{grad}'v(u) \in L^1_{\bf}(U).\]
\end{definition}
By an argument similar to the proof of \cite[Theorem $4.11$]{Honda14b}, we have the following:
\begin{proposition}\label{gradsov}
Under the same setting as in Definition \ref{poissondefinition}, let $R>0$, let $x \in X$, and let $f \in \mathcal{D}^2_{\bfC}(\Delta, B_R(x))$ with
\[\int_{B_R(x)}\left(|f|^2+|\Delta f|^2\right)dH^n\le L.\]
Then for any $r<R$, $|\overline{\partial}f|^2, |\partial f|^2 \in H^{1, 2n/(2n-1)}_{\bfC}(B_r(x))$ with
\[\int_{B_r(x)}\left(\left|\mathrm{grad}|\overline{\partial}f|^2\right|^{2n/(2n-1)}+ \left|\mathrm{grad}|\partial f|^2\right|^{2n/(2n-1)}\right)dH^n \le C(n, K, L, r, R).\]
Moreover, for any $u, v \in \mathcal{D}^2_{\bfC}(\Delta, B_R(x))$ with 
\[\int_{B_R(x)}\left(|u|^2+ |v|^2+|\Delta u|^2+ |\Delta v|^2\right)dH^n\le L,\]
we have $\{u, v\} \in H^{1, 2n/(2n-1)}_{\bfC}(B_r(x))$ with
\[\int_{B_r(x)}\left| \mathrm{grad} \{u, v\}\right|^{2n/(2n-1)}dH^n \le C(n, K, L, r, R)\]
for any $r<R$.
\end{proposition}
\begin{remark}\label{rem}
By Proposition \ref{gradsov} and \cite[Lemma $6.1$]{Honda14b}, for any $u, v \in \mathcal{D}^2_{\bfC}(\Delta, X)$, we see that $\{u, v\} \in H^{1, 2}_{\bfC}(X)$ holds if and only if $\mathrm{grad} \{u, v\} \in L^2_{\bfC}(T_{\bfC}X)$ holds.
\end{remark}

We need the following:
\begin{proposition}\label{eigenformula}
Let $M$ be a Fano manifold and let $u, v \in C^{\infty}_{\bfC}(M)$.
\begin{enumerate}
\item[(1)]\label{gradPoisson} For the gradient of the Poisson bracket we have
$$
 \mathrm{grad}'\{u, v\}
=[\mathrm{grad}'u, \mathrm{grad}'v] + \nabla_{\mathrm{grad}''v}\mathrm{grad}'u - \nabla_{\mathrm{grad}''u }\mathrm{grad}'v,
$$
\item[(2)]\label{deltaPoisson} If $\Delta^F_{\barpartial}u=\lambda u$ 
and $\Delta^F_{\barpartial}v=\nu v$, then
\begin{eqnarray*}
\Delta^F_{\barpartial}\{u, v\}
&=&(\lambda+\nu -1)\{u, v\} \\
& &- g_M(\nabla ''\mathrm{grad}'u,\nabla'\mathrm{grad}'' v) 
+ g_M(\nabla ''\mathrm{grad}'v, \nabla'\mathrm{grad}''u).
\end{eqnarray*}
\noindent
Here, in the standard notation of tensor calculus,
\begin{eqnarray*}
g_M(\nabla ''\mathrm{grad}'u,\nabla'\mathrm{grad}'' v) 
&=& g_{i\barj}g^{k\barl}\,\nabla_\barl\nabla^i u \nabla_k \nabla^\barj v\\
&=& \nabla_\barl\nabla^i u \nabla^\barl \nabla_i v.
\end{eqnarray*}
\end{enumerate}
\end{proposition}
\begin{proof} We choose a local holomorphic coordinates $z^1, \cdots, z^m$ and use the standard
notations of tensor calculus $\nabla^i = g^{i\barj}\nabla_\barj$ or $\nabla^\barj = g^{i\barj}\nabla_i$.
Then the Poisson bracket is written as
$$ \{u,v\} = \nabla^i u \nabla_i v - \nabla^i v \nabla_i u,$$
and the gradient vector field of type (1,0) is written as 
$$ \grad'u = \nabla^i u\frac{\partial}{\partial z^i}.$$
Since $\nabla_A\nabla_B u = \nabla_B\nabla_A u$ for functions $u$ we have
\begin{eqnarray*}
\nabla^i\{u,v\} &=& \nabla^i(\nabla^j u \nabla_j v - \nabla^j v \nabla_j u) \\
&=& \nabla^j u \nabla_j\nabla^i v - \nabla^j v \nabla_j\nabla^i u 
+ \nabla^\barj v \nabla_\barj \nabla^i u - \nabla^\barj u \nabla_\barj \nabla^i v 
\end{eqnarray*}
The last term is equal to the $i$-th component of 
$$ [\mathrm{grad}'u, \mathrm{grad}'v] + \nabla_{\mathrm{grad}''v}\mathrm{grad}'u - \nabla_{\mathrm{grad}''u }\mathrm{grad}'v.$$
This proves (1).  To prove (2) we first compute
\begin{eqnarray*}
\Delta \{u,v\} &=& -\nabla_k\nabla^k (\nabla^j u \nabla_j v - \nabla^j v \nabla_j u)\\
&=& -\nabla^k\nabla^j u \nabla_k\nabla_j v -\nabla_k\nabla^j u \nabla^k\nabla_j v \\
& & \qquad- \nabla_k\nabla^k\nabla^j u \nabla_j v- \nabla^j u \nabla_k\nabla^k\nabla_j v\\
& & \qquad  + \nabla^k\nabla^j v \nabla_k\nabla_j u + \nabla_k\nabla^j v \nabla^k\nabla_j u \\
& & \qquad + \nabla_k\nabla^k\nabla^j v \nabla_j u- \nabla^j v \nabla_k\nabla^k\nabla_j u.
\end{eqnarray*}
Then using the Ricci identity
$$ \nabla_k\nabla^j\nabla^k u= \nabla^j\nabla_k\nabla^k u  + R_k{}^{jk}{}_i\nabla^i u,$$
the definition of the Ricci curvature 
 $$R_k{}^{jk}{}_i\nabla^j = R_i{}^j = g^{j\barl}R_{i\barl},$$
 and the definition of the Ricci potential $F$
 $$ R_{i\barl} = g_{i\barl} + \nabla_i\nabla_\barl F, $$
 one can see that miraculous cancellations occur to obtain
 \begin{eqnarray*}
 \Delta^F_\barpartial \{u,v\} &=& \Delta\{u,v\} - \nabla^k\{u,v\}\nabla_k F \\
 &=& (\lambda + \mu - 1)\{u,v\} - \nabla_\barl\nabla^j u \nabla^\barl\nabla_j v 
 + \nabla_\barl\nabla^j v \nabla^\barl\nabla_j u.
 \end{eqnarray*}
 This completes the proof of (2).
\end{proof}

Recall that by Propositions \ref{spectralconv} and \ref{Laplacianweakconv}, for any Fano-Ricci limit space $(X, g_X, J, F)$ of $(X_i, g_{X_i}, J_i, F_i)$ and $\alpha$-eigenfunction $u \in \mathcal{D}^2_{\bfC}(\Delta^F_{\barpartial}, X)$ of $\Delta^F_{\barpartial}$, there exist sequences $\lambda_i \to \alpha$ and $u_i \in C^{\infty}_{\bfC}(X_i)$ such that $\Delta_{\barpartial}^{F_i}u_i=\lambda_i u_i$ and that $u_i$
and $du_i$ $L^2$-converge strongly to $u$ and$ du$ on $X$ respectively.
We call $(\lambda_i, u_i)$ \textit{a spectral approximation of $u$ (with respect to $(X_i, g_{X_i}, J_i, F_i)$)}.
Moreover if 
\[\sup_i ||h_{X_i}(\barpartial u_i, \barpartial F_i)||_{L^2} < \infty\]
holds, 
then $(\lambda_i, u_i)$ is said to be \textit{compatible}.

We first discuss a closedness of the Poisson bracket $\{\cdot, \cdot\}$ on $\Lambda_1$.
Recall that from \cite[Theorem $1.9$]{Honda14b} with Proposition \ref{Laplacianreal}, for every $u \in \mathcal{D}^2_{\bfC}(\Delta, X)$, we see that $u$ is twice differentiable on $X$ in the sense of \cite{Honda14a}. 
In particular,
\[[\mathrm{grad}'u, \mathrm{grad}'v]\]
is well-defined a.e. $x \in X$ for any $u, v \in \mathcal{D}^2_{\bfC}(\Delta, X)$.
\begin{proposition}\label{limitbracket}
Let $(X, g_X, J, F)$ be the Fano-Ricci limit space of the sequence $(X_i, g_{X_i}, J_i, F_i)$ and
let $u, v \in \Lambda_1(X)$. 
Assume that there exist compatible spectral approximations $(\lambda_i, u_i)$ and $(\nu_i, v_i)$ of $u$ and $v$ respectively.
Then we see that $u, v \in \mathcal{D}^2_{\bfC}(\Delta, X)$ and that
\[\{u, v\} \in \mathcal{D}_{\bfC}^{p_n, p_n}(\Delta^F_{\barpartial}, X)\]
with 
\begin{align}\label{homomo}
\mathrm{grad}'\{u, v\}=[\mathrm{grad}'u, \mathrm{grad}'v]
\end{align}
and
\[\Delta^F_{\barpartial}\{u, v\}=\{u, v\},\]
where $p_n=2n/(2n-1)$ (see Remark \ref{rem} for the definition of $\mathcal{D}^{p, q}_{\bfC}(\Delta^F_{\barpartial}, X)$).
In particular, the following are equivalent:
\begin{enumerate}
\item $\{u, v\} \in \Lambda_1(X)$.
\item $\mathrm{grad} \{u, v\} \in L^2_{\bfC}(T_{\bfC}X)$.
\end{enumerate}
\end{proposition}
\begin{proof}
By Propositions \ref{Laplacianreal}, \ref{Laplacianexplicite}, \ref{Laplacianequal}, and \cite[Theorem $4.13$]{Honda14b}, we see that $u, v \in \mathcal{D}^2_{\bfC}(\Delta, X)$, that $\{u, v\} \in H^{1, p_n}_{\bfC}(X)$, and that
$\mathrm{Hess}_{u_i}$ and $\mathrm{Hess}_{v_i}$ $L^2$-converge weakly to 
$\mathrm{Hess}_u$ and $\mathrm{Hess}_v$ on $X$ respectively.
Propositions \ref{Rellich} and \ref{gradsov} yield that $\{u_i, v_i\}$ $L^{p_n}$-converges strongly to $\{u, v\}$ on $X$, and that $d\{u_i, v_i\}$ $L^{p_n}$-converges weakly to $d\{u, v\}$ on $X$.

On the other hand, the Weitzenb\"ock formula on a Fano manifold yields
\[(\lambda_i-1)\int_{X_i}|u_i|^2dH^n_{F_i}=\int_{X_i}|\nabla''\mathrm{grad}'u_i|^2dH^n_{F_i}.\]
Thus letting $i \to \infty$ with this gives that $\nabla''\mathrm{grad}'u_i$ $L^2$-converges strongly to $0$ on $X$.
Similarly $\nabla''\mathrm{grad}'v_i$ $L^2$-converges strongly to $0$ on $X$.

Let $f \in \mathrm{LIP}_{\bfC}(X)$.
By (2.2e), there exists a sequence $f_i \in \mathrm{LIP}_{\bfC}(X)$ with $\sup_i||df_i||_{L^{\infty}}<\infty$ such that 
$f_i$ and $df_i$ $L^2$-converge strongly to $f$ and $df$ on $X$ respectively.

Since (2) of Proposition \ref{eigenformula} gives
\begin{align*}
&\int_{X_i}h_X(\barpartial \{u_i, v_i\}, \barpartial f_i)dH^n_{F_i} \\
&=\int_{X_i}(\lambda_i +\nu_i-1)\{u_i, v_i\}\overline{f_i}dH^n_{F_i} \\
& -\int_{X_i}\left(g_{X_i}(\overline{f_i} \nabla ''\mathrm{grad}'u_i, \nabla'\mathrm{grad}'' v_i) - g_X(\overline{f_i} \nabla ''\mathrm{grad}'v_i, \nabla'\mathrm{grad}'' u_i)\right) dH^n_{F_i},
\end{align*}
letting $i \to \infty$ yields that $\{u, v\} \in \mathcal{D}_{\bfC}^{p_n, p_n}(\Delta^F_{\barpartial}, X)$ with 
$\Delta^F_{\barpartial}\{u, v\}=\{u, v\}$.

On the other hand, since
\[[\mathrm{grad}'u_i, \mathrm{grad}'v_i]=\nabla_{\mathrm{grad}'u_i}\mathrm{grad}'v_i-\nabla_{\mathrm{grad}'v_i}\mathrm{grad}'u_i,\]
by \cite[Theorem $4.13$]{Honda14b} we see that $[\mathrm{grad}'u_i, \mathrm{grad}'v_i]$ $L^{2n/(2n-1)}$-converges weakly to $[\mathrm{grad}'u, \mathrm{grad}'v]$ on $X$.
Therefore by letting $i \to \infty$ in (1) of Proposition \ref{eigenformula}, we have (\ref{homomo}).
The final equivalence follows from Remark \ref{rem}.
\end{proof}
\begin{corollary}\label{lielambda}
Let $(X, g_X, J, F)$ be the Fano-Ricci limit space of the sequence $(X_i, g_{X_i}, J_i, F_i)$ and let $u, v \in \Lambda_1(X)$.
Assume that there exist compatible spectral approximations $(\lambda_i, u_i)$ 
and $(\nu_i, v_i)$ of $u$ and $v$, respectively such that
\begin{align}\label{L2gradestimate}
\sup_i||\mathrm{grad}'\{u_i, v_i\}||_{L^2}<\infty.
\end{align}
Then $\{u, v\} \in \Lambda_1(X)$. 
\end{corollary}
\begin{proof}
This is a direct consequence of Theorem \ref{Rellich}, Propositions \ref{L2norm} and (the proof of) \ref{limitbracket}.
\end{proof}
\begin{corollary}
Let $(X, g_X, J, F)$ be the Fano-Ricci limit space of the sequence $(X_i, g_{X_i}, J_i, F_i)$.
Assume 
\[\sup_i||\barpartial F_i||_{L^{\infty}}<\infty.\]
Moreover we assume that one of the following holds:
\begin{enumerate}
\item[(1)] $\Lambda_1(X) \subset \mathrm{LIP}_{\bfC}(X)$.
\item[(2)] $F \equiv 0$.
\end{enumerate}
Then we have the following closedness of $\Lambda_1(X)$ for the Poisson bracket:
\[\{u, v\} \in \Lambda_1(X)\]
for any $u, v \in \Lambda_1(X)$. 
\end{corollary}
\begin{proof}
Assume that (1) holds.
Let $u, v \in \Lambda_1(X)$.
Then by the proof of Proposition \ref{limitbracket} and \cite[Theorem $4.11$]{Honda14b} we have $\mathrm{Hess}_u, \mathrm{Hess}_v \in L^2_{\bfC}(T_{\bfC}^*X \otimes T^*_{\bfC}X)$.
In particular the assumption (1) implies
\[\nabla \{u, v\} \in L^2_{\bfC}(T_{\bfC}X).\]
Therefore Proposition \ref{limitbracket} gives our closedness.

On the other hand, by \cite[Theorem $7.9$]{CheegerColding3}, since if (2) holds, then (1) holds, this completes the proof.
\end{proof}

The rest of this subsection is devoted to 
a construction of a subspace $\Lambda$ of $\Lambda_1$ which is a Lie algebra by the Poisson bracket $\{\cdot, \cdot \}$.
Note that, on almost smooth setting, we will see in Section $5$ that if a subspace of $\Lambda_1$ is closed with respect to the Poisson bracket $\{\cdot, \cdot \}$, then it becomes a Lie algebra, automatically.
See Proposition \ref{eig}.

For this purpose we need the following definition:
\begin{definition}
Let $(X, g_X, J, F)$ be the Fano-Ricci limit space of the sequence
$(X_i, g_{X_i}, J_i, F_i)$.
\begin{enumerate}
\item A function $u \in \Lambda_1(X)$ is said to be \textit{a (compatible) limit 1-eigenfunction} if there exists a (compatible) spectral approximation $(1, u_i)$ of $u$. 
\item A subspace $\Lambda$ of $\Lambda_1(X)$ is said to be \textit{the limit 1-eigenspace} if every $u \in \Lambda$ is a limit 1-eigenfunction with
\[\mathrm{dim}\,\Lambda=\lim_{i \to \infty}\mathrm{dim}\,\Lambda_1(X_i).\]
Moreover if every $u \in \Lambda$ is a compatible limit 1-eigenfunction, then $\Lambda$ is said to be \textit{compatible}.
\end{enumerate}
\end{definition}
Since it is easy to check that the limit 1-eigenspace is unique if it exists, we denote it by $\lim_{i \to \infty}\Lambda_1(X_i)$.
In general, for a subspace $\Lambda$ of $\Lambda_1$, let $\mathfrak{h}^{\Lambda}(X):=\Phi(\Lambda )$, where $\Phi$ is defined in Remark \ref{eigenhol}.
Roughly speaking, the following means that $\mathfrak{h}^{\Lambda}(X)$ is the space of $L^2$-strong limits of holomorphic vector fields on $X_i$ if $\Lambda=\lim_{i \to \infty}\Lambda_1(X_i)$.
\begin{proposition}
Let $(X, g_X, J, F)$ be the Fano-Ricci limit space of the sequence
$(X_i, g_{X_i}, J_i, F_i)$.
\begin{enumerate}
\item[(1)] There exist a subsequence $i(j)$ and the limit 1-eigenspace of $\Lambda_1(X_{i(j)})$.
\item[(2)] If
\[\sup_i||\barpartial F_i||_{L^{\infty}}<\infty,\]
then any spectral approximations are compatible.
In particular, the limit 1-eigenspace of $\Lambda_1(X_{i})$ is compatible.
\item[(3)] If $\Lambda$ is the limit 1-eigenspace of $\Lambda_1(X_i)$, then
\[\lim_{i \to \infty}||\mathcal{F}_{X_{i}}||_{\mathrm{op}}=||\mathcal{F}_X|_{\Lambda}||_{\mathrm{op}}.\]
In particular,
\[\lim_{i \to \infty}||\mathcal{F}_{X_i}||_{\mathrm{op}}=||\mathcal{F}_X||_{\mathrm{op}}\]
holds if and only if
\[\mathcal{F}_X \equiv 0\]
on $(\mathfrak{h}^{\Lambda}(X))^{\bot}$, where $(\mathfrak{h}^{\Lambda}(X))^{\bot}$ is the orthogonal complement of $\mathfrak{h}^{\Lambda}(X)$ with respect to the $L^2$-norm.
\item[(4)]  If
\[\lim_{i \to \infty}\mathfrak{h}_1(X_i)=\mathfrak{h}_1(X),\]
then the limit 1-eigenspace coincides with $\Lambda_1(X)$.
\end{enumerate}
\end{proposition}
\begin{proof}
(1) is a direct consequence of Proposition \ref{uppersemi}.
By Proposition \ref{Laplacianweakconv} and the definition of the limit 1-eigenspace, (2) and (4) are  trivial.
The proof of (2) of Proposition \ref{Futaki3} yields (3).
Thus this completes the proof.
\end{proof}
In order to get an $L^2$-estimate (\ref{L2gradestimate}) for spectral approximations, we prepare the following.
\begin{proposition}\label{LpL2}
Let $(X, g_X, J, F)$ be a Fano-Ricci limit space with $H^n(X) \ge v$, $\mathrm{diam}\,X \le d$ and $F \ge c$, and let $V \in \mathfrak{h}_1(X)$ with
\[||V||_{L^p} \le L\]
for some $p \in (1, 2)$.
Then 
\[||V||_{L^2}\le C(n, K, d, v, c, L, p).\]
\end{proposition}
\begin{proof}
The proof is done by a contradiction.
Assume that the assertion is false.
Then  there exist a sequence of $(n, K)$-Fano-Ricci limit spaces $(X_i, J_i, g_{X_i}, F_i)$ with $H^n(X_i)\ge v$, $\mathrm{diam}\,X_i \le d$, $F_i \ge c$,  and a sequence of $V_i \in \mathfrak{h}_1(X_i)$ with 
\[\sup_i||V_i||_{L^p}<\infty\]
and
\begin{align}\label{L2infty}
\lim_{i \to \infty}||V_i||_{L^2} = \infty.
\end{align}
By (2.2a) and Corollary \ref{compactness}, we can assume 
without loss of generality that there exist the noncollapsed Gromov-Hausdorff limit $X$ of $X_i$, 
the $L^2$-strong limit $J$ of $J_{i}$ on $X$,  
the $L^2$-strong limit $F \in H^{1, 2}(X) \cap L^{\infty}(X)$ of $F_{i}$ on $X$, and the $L^p$-weak limit $V$ of $V_i$ on $X$.
Let $W_i:=||V_i||_{L^2}^{-1}V_i \in \mathfrak{h}_1(X_i)$. 
From (\ref{L2infty}), we see that $W_i$ $L^p$-converges weakly to $0$ on $X$.
Since $||W_i||_{L^2}=1$, this is an $L^2$-weak convergence.
On the other hand, by Proposition \ref{uppersemi}, there exist a subsequence $W_{i(j)}$ and the $L^2$-strong limit $W \in \mathfrak{h}_1(X)$.
In particular, $||W||_{L^2}=1$. This is a contradiction.
\end{proof}
\begin{corollary}
Let $(X, g_X, J, F)$ be the Fano-Ricci limit space of $(X_i, g_{X_i}, J_i, F_i)$.
Then
\[\lim_{i \to \infty}\mathrm{dim}\,\mathfrak{h}_1(X_i)=\mathrm{dim}\,\mathfrak{h}_1(X)\] 
holds if and only if 
for any $V \in \mathfrak{h}_1(X)$ and subsequence $\{i(j)\}_j$, there exist a subsequence $\{j(k)\}_k$ of $\{i(j)\}_j$, $p \in (1, 2)$ and a sequence $V_{j(k)} \in \mathfrak{h}_1(X_{j(k)})$ such that $V_{j(k)}$ $L^p$-converges weakly to $V$ on $X$.
\end{corollary}
\begin{proof}
This is a direct consequence of Propositions \ref{uppersemi} and \ref{LpL2}.
\end{proof}
\begin{corollary}\label{L2est}
Let $M$ be a Fano manifold with $\mathrm{Ric}_M \ge K$, $\mathrm{diam}\,M\le d$, and $H^n(M) \ge v$,  let $F$ be the Ricci potential with the canonical normalization with $F \ge c$, and let $u, v \in \Lambda_1$ with 
\[||h_M(\barpartial u, \barpartial F)||_{L^2}+||h_M(\barpartial v, \barpartial F)||_{L^2}\le L.\]
Then
\[||\{u, v\}||_{H^{1, 2}_{\bfC}} \le C(n, K, d, v, c, L).\]
In particular if $(M, g_M, J, F)$ is a K\"ahler-Ricci soliton, i.e. $F \in \Lambda_1(M)$, and if any $\alpha$-eigenfunction $w \in \Lambda_1(M)$ of the action $-F$ on $\Lambda_1(M)$ defined by the Poisson bracket $\{ \cdot, \cdot \}$, i.e.
\[-\{F, w\}=\alpha w\]
with
\[||\barpartial F||_{L^4}+||h_M(\barpartial w, \barpartial F)||_{L^2} \le L,\]
then we have
\[\alpha \le C(n, K, d, v, c, L).\]
\end{corollary}
\begin{proof}
Propositions \ref{Laplacianexplicite}, \ref{gradsov} and \ref{Laplacianequal} yield 
\[||\mathrm{grad}'\{u, v\}||_{L^{2n/(2n-1)}} \le C(n, K, d, v, c, L).\]
Since $\mathrm{grad}'\{u, v\} \in \mathfrak{h}_1(M)$ (see for instance Remark \ref{pointwise}),  Proposition \ref{LpL2} yields
\[||\mathrm{grad}'\{u, v\}||_{L^2} \le C(n, K, d, v, c, L).\]
Thus the assertion follows from this, Propositions \ref{L2norm} and \ref{Poincare}
\end{proof}
\begin{remark}
In Corollary \ref{L2est}, by \cite[Theorem $1.2$]{PhongSongSturm} with Corollary \ref{compactness}, we drop the assumption of $L^4$-bound on $\barpartial F$.
In fact, we can get
\[||\barpartial F||_{L^{\infty}}\le C(n, K, d, v, c),\]
automatically.
See Theorem \ref{decompconti}.
\end{remark}
\begin{proposition}\label{poissonstr}
Let $(X, g_X, J, F)$ be the Fano-Ricci limit space of the sequence
$(X_i, g_{X_i}, J_i, F_i)$.
Then we have the following:
\begin{enumerate}
\item[(1)] Let $u, v \in \Lambda_1(X)$ be compatible limit 1-eigenfunctions.
Then we see that $\{u, v\} \in \Lambda_1(X)$, and that $\mathcal{F}_X(\mathrm{grad}'\{u, v\}, g_X)=0$.
\item[(2)]If $\Lambda:=\lim_{i \to \infty}\Lambda_1(X_i)$ is compatible, then
$(\Lambda, \{\cdot, \cdot\})$ and $(\mathfrak{h}^{\Lambda}(X), [ \cdot, \cdot ])$ are Lie algebras, and $\mathcal{F}_X|_{\mathfrak{h}^{\Lambda}(X)}$ is a character of $\mathfrak{h}^{\Lambda}(X)$ as a Lie algebra.
Moreover the map $\Psi_{\Lambda}:\Lambda \to \mathfrak{h}^{\Lambda}(X)$  defined by the restriction of $\Psi$ to $\Lambda$, i.e.
\[\Psi_{\Lambda}(u):=\mathrm{grad}'u,\]
gives an isomorphism between them as Lie algebras.
\end{enumerate}
\end{proposition}
\begin{proof}
We first prove (1).
Let $(1, u_i), (1, v_i)$ be compatible spectral approximations of $u, v$, respectively. 
Propositions \ref{Rellich}, \ref{Laplacianweakconv} and Corollary \ref{L2est} yield that
$\{u, v\} \in \Lambda_1$, and that $\{u_i, v_i\}$ and $d\{u_i, v_i\}$ $L^2$-converge strongly to $\{u, v\}$ and $d\{u, v\}$ on $X$, respectively.
In particular since the Futaki invariant is a character as a Lie algebra on smooth setting, (2) of Proposition \ref{Futaki1} yields
\[\mathcal{F}_X(\mathrm{grad}'\{ u, v\}, g_X)=\lim_{i \to \infty}\mathcal{F}_{X_i}(\mathrm{grad}'\{ u_i, v_i\}, g_{X_i})=0.\]
This completes the proof of (1).

We turn to the proof of (2).
By Proposition \ref{limitbracket} and (1), it suffices to check the Jacobi identity for the Poisson bracket $\{\cdot, \cdot \}$.
Let $u, v, w \in \Lambda$ and let $(1, u_i), (1, v_i)$, and $(1, w_i)$ be compatible spectral approximations of $u, v$, and $w$.
Since
\[\{u_i, \{ v_i, w_i\} \}+\{w_i, \{ u_i, v_i\} \}+\{v_i, \{ w_i, u_i\} \}=0,\]
letting $i \to \infty$ with Proposition \ref{uppersemi} and the proof of (1) gives the Jacobi identity for the Poisson bracket $\{\cdot, \cdot \}$.
\end{proof}
By Propositions \ref{lowerbound}, \ref{poissonstr} and Remark \ref{holderlipschitz}, we have the following compactness:
\begin{corollary}\label{liestr}
Let $(X_i, g_{X_i}, J_i, F_i)$ be a sequence of Fano manifolds with $\mathrm{Ric}_{X_i}\ge K$, $H^n(X_i) \ge v$, $\mathrm{diam}\,X_i\le d$, and 
\[\sup_i\int_{X_i}|s_{X_i}|^q\,dH^n<\infty\]
for some $q>n$.
Then there exist a subsequence $i(j)$, the Fano-Ricci limit space $(X, g_X, J, F)$ of $(X_{i(j)}, g_{X_{i(j)}}, J_{i(j)}, F_{i(j)})$ and the  compatible limit 1-eigenspace $\Lambda:=\lim_{j \to \infty}\Lambda_1(X_{i(j)})$ such that
$(\Lambda, \{\cdot, \cdot \})$ and $(\mathfrak{h}^{\Lambda}(X), [\cdot, \cdot ])$ are finite dimensional Lie algebra.
Moreover the map
\[\Psi_{\Lambda}:\Lambda \to \mathfrak{h}^{\Lambda}(X)\]
defined by $\Psi_{\Lambda}(u):=\mathrm{grad}'u$ gives an isomorphism between them as Lie algebras. 
Furthermore, $\mathcal{F}_X|_{\mathfrak{h}^{\Lambda}}$ is a character of $\mathfrak{h}^{\Lambda}(X)$ as a Lie algebra. 
\end{corollary}
In particular we have the following:
\begin{corollary}\label{boundedricci}
Let $(X_i, g_{X_i}, J_i, F_i)$ be a sequence of Fano manifolds with $H^n(X_i) \ge v$, $\mathrm{diam}\,X_i\le d$, and 
\[|\mathrm{Ric}_{X_i}| \le K.\]
Then the same conclusion as in Corollary \ref{liestr} holds.
\end{corollary}
It is worth pointing out that in the setting of Corollary \ref{boundedricci} we can prove that $F$ is the Ricci potential of $(X, g_X, J)$ in some weak sense. See \cite{Honda15}.
We will discuss again similar results as above in almost smooth setting in Section 5.
\section{Almost smooth Fano-Ricci limit space}
\subsection{Decomposition theorem on an almost smooth Fano-Ricci limit space}
Recall that a Fano-Ricci limit space $(X, g_X, J, F)$ 
is the limit space of $(X_i, g_{X_i}, J_i, F_i)$ 
satisfying (2.1a) - (2.1e), (4.1a) and (4.1b). We say that $(X, g_X, J, F)$ is an almost smooth Fano-Ricci
limit space if in addition the conditions (5.1a) - (5.1c) below are satisfied.

\begin{enumerate}
\item[(5.1a)]\label{regular} There exists an open (dense) subset $\mathcal{R}$ of $X$ such that $H^n(X \setminus \mathcal{R})=0$,  that $(\mathcal{R}, g_X|_{\mathcal{R}}, J|_{\mathcal{R}})$ is a smooth K\"ahler manifold, and that $F|_{\mathcal{R}} \in C^{\infty}(\mathcal{R})$ with 
\[\mathrm{Ric}_{\omega}-\omega= \sqrt{-1}\partial \overline{\partial}F\]
on $\mathcal{R}$.
\item[(5.1b)]\label{Liuville} Every $L^2$-holomorphic function on $\mathcal{R}$ is constant.
\item[(5.1c)]\label{hamiltonian} We have
\[\{u \in H^{1, 2}_{\bfC}(X); \mathrm{grad}'u|_{\mathcal{R}} \in \mathfrak{h}_{\mathrm{reg}}(X)\} \subset \mathcal{D}^2_{\bfC}(\Delta^F_{\overline{\partial}}, X),\]
where $\mathfrak{h}_{\mathrm{reg}}(X)$ is the set of $L^2$-holomorphic vector fields on $\mathcal{R}$, or 
equivalently on  $X$ by the assumption (5.1a),  having smooth potentials on $\mathcal{R}$.
\end{enumerate}
Note that by Corollary \ref{compactness}, we have $F \in H^{1, 2}(X)$.
Recall 
$$\Lambda_1=\{f \in \mathcal{D}^2_{\bfC}(\Delta_{\overline{\partial}}^F, X); \Delta^F_{\overline{\partial}}f
=f\}. $$
Let $\Lambda$ be a complex subspace of $\Lambda_1$,
$\mathfrak{h}^{\Lambda}(X)$ the set of $V \in \mathfrak{h}_{\mathrm{reg}}(X)$ with $V=\mathrm{grad}'u$ for some $u \in \Lambda$ (i.e. $\mathfrak{h}^{\Lambda}(X)=\Phi (\Lambda)$), and $\tilde{\mathfrak{h}}(X)$ the set of $V \in \mathfrak{h}_{reg}(X)$ with $V=\mathrm{grad}'u$ for some $u \in H^{1, 2}_{\bfC}(X)$.
Note that $\mathfrak{h}^{\Lambda_1}(X)=\mathfrak{h}_1(X)$.

We remark the following:
\begin{proposition}\label{h1h}
We have
\[\mathfrak{h}_{\mathrm{reg}}(X)=\mathfrak{h}_1(X).\]
\end{proposition}
\begin{proof}
Let $V \in \mathfrak{h}_{\mathrm{reg}}(X)$.
Then there exists a $\bfC$-valued smooth function $f$ on $\mathcal{R}$ such that $V=\mathrm{grad}'f$ on $\mathcal{R}$.
By (1) of Proposition \ref{suffsobo}, there exists $u \in H^{1, 2}_{\bfC}(X)$ such that $\mathrm{grad}'f=\mathrm{grad}'u$ on $\mathcal{R}$.
Thus, by the assumption (5.1c), $V \in \mathfrak{h}_1(X)$.
This completes the proof. 
\end{proof}

\begin{remark}\label{pointwise}
By a simple calculation we have the following:
\begin{enumerate}
\item[(1)] We have
\[\mathrm{grad}'\{u, v\} = [ \mathrm{grad}'u, \mathrm{grad}'v]\]
on $\mathcal{R}$ for any $u, v \in \Lambda_1$. In particular by Corollary \ref{1steigen}, $\mathrm{grad}'\{u, v\}$ is a holomorphic vector field on $\mathcal{R}$.
\item[(2)] If a smooth function $u$ on $\mathcal{R}$ satisfies that $\mathrm{grad}'u$ is a holomorphic vector field on $\mathcal{R}$, then 
\[\overline{\partial}(\Delta^F_{\overline{\partial}}u-u)=0\]
on $\mathcal{R}$.
\end{enumerate}
\end{remark}

\begin{proposition}\label{eig}
If $u \in H^{1, 2}_{\bfC}(X)$ satisfies $\mathrm{grad}'u \in \mathfrak{h}_{\mathrm{reg}}(X)$ and
\[\int_X\,u\,dH^n_F=0,\]
then $u \in \Lambda_1$.
In other words, $\tilde{\mathfrak{h}}(X)=\mathfrak{h}^{\Lambda_1}(X)(=\mathfrak{h}_1(X))$.
\end{proposition}
\begin{proof}
Let $u \in H^{1, 2}_{\bfC}(X)$ with $\mathrm{grad}'u \in \mathfrak{h}_{\mathrm{reg}}(X)$.
Then (5.1c) gives $u \in \mathcal{D}^2_{\bfC}(\Delta^F_{\overline{\partial}}, X)$.
In particular $\Delta^F_{\overline{\partial}}u \in L^2_{\bfC}(X)$.

Thus by (2) of Remark \ref{pointwise} and (5.1b), we see that $\Delta^F_{\overline{\partial}}u-u$ is constant.
Proposition \ref{Poisson} yields $\Delta^F_{\overline{\partial}}u-u=0$.
This completes the proof.
\end{proof}
\begin{proposition}\label{isometry}
Assume that for any $u, v \in \Lambda$, $\{u, v\} \in \Lambda$.
Then
$(\Lambda, \{\cdot, \cdot \})$ and $(\mathfrak{h}^{\Lambda}(X), [ \cdot, \cdot])$ are finite dimensional complex Lie algebras.
Moreover the map $\Psi_{\Lambda}: \Lambda \to \mathfrak{h}^{\Lambda}(X)$ defined by
\[\Psi_{\Lambda}(u):=\mathrm{grad}'u\]
gives an isomorphism between them as Lie algebras.
\end{proposition}
\begin{proof}
This is a direct consequence of (1) of Remark \ref{pointwise}.
\end{proof}

\subsection{K\"ahler-Ricci limit solitons.}

Let 
$(X, g_X, J, F)$ be an almost smooth Fano-Ricci
limit space, that is, the conditions (2.1a) - (2.1e), (4.1a), (4.1b), (5.1a) - (5.1c) are satisfied.

\begin{proposition}\label{killingequiv}
Let $u \in \Lambda_1$.
Then the following are equivalent:
\begin{enumerate}
\item[(1)] $\mathrm{Re}(\mathrm{grad}'u)$ is a Killing vector field on $\mathcal{R}$, where $\mathrm{Re}(\mathrm{grad}'u)$ is the real part of $\mathrm{grad}'u$.
\item[(2)] $\mathrm{Re}(u)$ is constant.
\end{enumerate}
\end{proposition}
\begin{proof}
By a simple calculation we have
\begin{align}\label{lie}
L_{\mathrm{Re}(\mathrm{grad}'u)}\omega_X= \sqrt{-1}\partial \overline{\partial}\mathrm{Re}(u)
\end{align}
on $\mathcal{R}$.

Assume that $\mathrm{Re}(\mathrm{grad}'u)$ is a Killing vector field on $\mathcal{R}$.
By taking the (complex) trace of (\ref{lie}) we have $\Delta_{\overline{\partial}}\mathrm{Re}(u)=0$ on $\mathcal{R}$.
Thus Proposition \ref{lapdom} shows that $\mathrm{Re}(u)$ is constant.

By (\ref{lie}), the converse is trivial.
This completes the proof.
\end{proof}
\begin{definition}[K\"ahler-Ricci limit soliton]
We say that an almost smooth Fano-Ricci limit space \textit{$(X, g_X, J, F)$ is a K\"ahler-Ricci limit soliton} if $\mathrm{grad}'F \in \mathfrak{h}_{reg}(X)$.
\end{definition}
Note that by Proposition \ref{eig}, $(X, g_X, J, F)$ is a K\"ahler-Ricci limit soliton if and only if $F \in \Lambda_1$ holds. Further, by Proposition \ref{killingequiv}, $\mathrm{Re}(\mathrm{grad}'(iF))$ is a 
Killing vector field.

\begin{theorem}[Decomposition theorem]\label{decomp} 
Let $(X, g_X, J, F)$ be a K\"ahler-Ricci limit soliton. 
For a complex subspace $\Lambda$ of $\Lambda_1$, we assume the following:
\begin{enumerate}
\item[(1)] For any $u, v \in \Lambda$, $\{u, v\} \in \Lambda$.
\item[(2)] For every $u \in \Lambda$, $\{u, F\} \in \Lambda$.
\end{enumerate}
Then $-\mathrm{grad}'F$ acts on $\mathfrak{h}^{\Lambda}(X)$ by the adjoint action and $\mathfrak{h}^{\Lambda}(X)$ has a decomposition
\[\mathfrak{h}^{\Lambda}(X)=\mathfrak{h}_0^{\Lambda}(X) \oplus \bigoplus_{\alpha >0}\mathfrak{h}^{\Lambda}_{\alpha}(X),\]
where $\mathfrak{h}^{\Lambda}_{\alpha}(X)$ is the $\alpha$-eigenspace of the adjoint action of $-\mathrm{grad}'F$.
Furthermore,  $\mathfrak{h}_0^{\Lambda}(X)$ is isomorphic as a Lie algebra to the complexification of a real Lie algebra 
$\tilde{\Lambda}:=\left\{ u \in \Lambda | u=-\overline{u}\right\}$ with the Poisson bracket  
$\{ \cdot, \cdot \}$, 
and $\bigoplus_{\alpha >0}\mathfrak{h}^{\Lambda}_{\alpha}(X)$ is nilpotent.
Moreover, the map
\[\phi_{\Lambda}:  \tilde{\Lambda} \to \mathcal{K}(\mathcal{R})\]
defined by $\phi_{\Lambda}(u):=2\mathrm{Re}(\mathrm{grad}'u)$ is an inclusion of Lie subalgebra, 
where $\mathcal{K}(\mathcal{R})$ is the space of all Killing vector fields on $\mathcal{R}$.
In particular, if $\mathcal{R}$ coincides with the regular set of $X$ and the image of $\phi_{\Lambda}$ is contained in the Lie algebra of the isometry group of $\mathcal{R}$, then $\mathfrak{h}_0^{\Lambda}$ is reductive.
\end{theorem}
\begin{proof}
For every $u \in \Lambda_1$, let
\[\overline{\Delta^F_{\overline{\partial}}}u:=\overline{\Delta^F_{\overline{\partial}}\overline{u}}\]
on $\mathcal{R}$.
Then by a simple calculation we have
\[\Delta^F_{\overline{\partial}}u-\overline{\Delta^F_{\overline{\partial}}}u=\{F, u\} \]
on $\mathcal{R}$.
In particular, $\overline{\Delta^F_{\overline{\partial}}}u \in H^{1, 2}_{\bfC}(X)$.
Thus Proposition \ref{lapdom} gives $\overline{u} \in \mathcal{D}^2_{\bfC}(\Delta^F_{\overline{\partial}}, X)$.

Therefore for every $\xi \in \mathfrak{h}^{\Lambda}_{\alpha}(X)$, by letting $u_{\xi}:=\Psi^{-1}_{\Lambda}(\xi)$, 
we have
\[\Delta^F_{\overline{\partial}}\overline{u_{\xi}}=(\alpha +1)\overline{u_{\xi}}.\]
Thus Corollary \ref{1steigen} gives $\alpha \ge 0$.
Therefore we have a decomposition
\[\mathfrak{h}^{\Lambda}(X)=\bigoplus_{\alpha \ge 0}\mathfrak{h}^{\Lambda}_{\alpha}(X).\]
From the argument above we see that $\Psi_{\Lambda}^{-1}(\mathfrak{h}^{\Lambda}_0(X))$ coincides with
\[\Lambda_{1, 0}:=\{u \in \Lambda ; \Delta^{F}_{\overline{\partial}}\overline{u} =\overline{u}\}.\]
In particular, for every $u \in \Lambda_{1, 0}$, we have $\mathrm{Re}(u), \mathrm{Im}(u) \in \Lambda_{1, 0}$.
It is easy to check that $\tilde{\Lambda}$ is a real Lie algebra and that $\Lambda_{0, 1}$ is isomorphic to the complexification of $\tilde{\Lambda}$.

On the other hand, from the Jacobi identity on $\mathcal{R}$, we have
\[[\mathfrak{h}_{\alpha}^{\Lambda}(X), \mathfrak{h}_{\beta}^{\Lambda}(X)] \subset \mathfrak{h}^{\Lambda}_{\alpha + \beta}(X)\]
for any $\alpha, \beta \ge 0$.
Since the dimension of $\mathfrak{h}^{\Lambda}(X)$ is finite, there exists a finite subset $\Gamma$ of $(0, \infty)$ such that $\mathfrak{h}^{\Lambda}_{\alpha}(X)=0$ for every $\alpha \in (0, \infty) \setminus \Gamma$.
This shows that $\bigoplus_{\alpha>0}\mathfrak{h}^{\Lambda}_{\alpha}(X)$ is nilpotent.

Next we prove that $\phi_{\Lambda}$ is embedding as Lie algebras.
Note that by Proposition \ref{killingequiv}, $\phi_{\Lambda}$ is well-defined. 
By a simple calculation, it is easy to check that $\phi_{\Lambda}$ is bracket preserving.
Let $u \in \tilde{\Lambda}$ with $\phi_{\Lambda}(u)=0$.
Since
\[\mathrm{Re}(\mathrm{grad}'u)=\frac{\mathrm{grad}'u-\mathrm{grad}''u}{2},\] 
we have $\mathrm{grad}'u=\mathrm{grad}''u$.
Thus $\mathrm{grad}'u=\mathrm{grad}''u=0$.
Proposition \ref{Poincare} gives that $u$ is constant.
Since
\[\int_Xu\,dH^n_F=\int_X\Delta^{F}_{\overline{\partial}}u\,dH^n_F=0,\]
we have $u=0$, i.e. $\phi_{\Lambda}$ is embedding as Lie algebras.

Finally we assume that $\mathcal{R}$ coincides with the regular set of $X$ and that the image of $\phi_{\Lambda}$ is contained in the Lie algebra $\mathfrak{g}$ of the isometry group $G$ of 
$\mathcal{R}$.
Note that $G$ is isomorphic to the isometry group of $X$ because all isometry $f: X \to X$ preserve the regular set (note that in this assumption, the regular set is open and convex. In particular, the distance function on $\mathcal{R}$ defined by the smooth Riemannian metric $g_X$ coincides with the restriction of $d_X$ to $\mathcal{R}$. See \cite[Theorem $3.7$]{CheegerColding2} or \cite[Theorem $1.2$]{ColdingNaber}).
Therefore, from \cite[Theorem $4.1$]{CheegerColding2}, $G$ is a compact Lie group.
Thus we see that $\mathfrak{h}_0^{\Lambda}(X)$ is reductive.
\end{proof}
\begin{remark}\label{completeremark}
Assume that $\mathcal{R}$ coincides with the regular set of $X$.
Then since the isometry group of $\mathcal{R}$ is isomorphic to that of $X$ and is a compact Lie group, we see that for every Killing vector field $V$ on $\mathcal{R}$, $V$ is in the Lie algebra of the isometry group of $\mathcal{R}$ if and only if $V$ is complete. 
\end{remark}
\begin{remark}
One of key points for the condition (5.1b) in the arguments above is the following:
\begin{enumerate}
\item[($\star$)] If $u \in H^{1, 2}(X)$ satisfies $\mathrm{grad}'u \in \mathfrak{h}_{reg}(X)$ and 
$ \int_X\,u\, dH^n_F = 0$, then $u \in \Lambda_1$.
\end{enumerate}
In fact if we replace (5.1b) by ($\star$), then we can prove the same results above.

It is worth pointing out that ($\star$) holds if the Weitzenb\"ock \textit{formula}
\begin{align}\label{formula}
\int_X|\Delta^F_{\overline{\partial}}f|^2\,dH^n_F = \int_X|\nabla '' \mathrm{grad}'f|^2\,dH^n_F+\int_X|\overline{\partial}f|^2\,dH^n_F.
\end{align}
holds for every $f \in \mathcal{D}^2_{\bfC}(\Delta_{\overline{\partial}}^F, X)$.
Note that by using a result in \cite{Honda15} we can establish (\ref{formula}) under an additional assumption:
\[\sup_i|\mathrm{Ric}_{X_i}|<\infty.\]
\end{remark}

\subsection{Remarks on the Lie algebra structure of $\Lambda_1$ on almost smooth setting}
Let $(X, g_X, J, F)$ be an almost smooth Fano-Ricci
limit space so that (2.1a) - (2.1e), (4.1a), (4.1b), (5.1a) - (5.1c) are satisfied. 
We add the following assumption:
\begin{enumerate}
\item[(5.3a)] The inclusion
\[H^{1, 2}_{c}(\mathcal{R}) \hookrightarrow H^{1, 2}(X)\]
is isomorphic.
\end{enumerate}
Then we can apply Proposition \ref{suffsobo} with $U = \mathcal{R}$.

Compare the following with Proposition \ref{limitbracket}.
\begin{proposition}\label{equivlie}
Let $(X, g_X, J, F)$ be an almost smooth Fano-Ricci
limit space. 
Then for any $u, v \in \Lambda_1$, the following are equivalent:
\begin{enumerate}
\item[(1)] $\{u, v\} \in \Lambda_1$.
\item[(2)] $\{u, v\} \in H^{1, 2}_{\bfC}(X)$.
\end{enumerate}
Moreover if (5.3a) holds, then these also are equivalent to the following:
\begin{enumerate}
\item[(3)] $\mathrm{grad}'\{u, v\} \in L^{2}_{\bfC}(T^\prime X)$.
\end{enumerate}
\end{proposition}
\begin{proof}
It is trivial that if (1) holds, then (2) holds.

Assume that (2) holds.
Then by (1) of Remark \ref{pointwise}, we have $\mathrm{grad}'\{u, v\} \in \mathfrak{h}_{reg}(X)$. 
In particular $\mathrm{grad}'\{u, v\}$ is a holomorphic vector field on $\mathcal{R}$.
Since $\{u, v\} \in H^{1, 2}_{\bfC}(X)$, we have $\mathrm{grad}'\{u, v\} \in L^2_{\bfC}(T'X)$.
Therefore by (5.1c), we have $\{u, v\} \in \mathcal{D}^2_{\bfC}(\Delta^F_{\overline{\partial}}, X)$.

By a simple calculation we have
\[\overline{\partial}(\Delta^F_{\overline{\partial}}\{u, v\}-\{u, v\})=0\]
on $\mathcal{R}$, i.e. $\Delta^F_{\overline{\partial}}\{u, v\}-\{u, v\} \in \mathfrak{h}_{reg}(X)$.
Thus by  (5.1b),  $\Delta^F_{\overline{\partial}}\{u, v\}-\{u, v\}$ is a constant function.

On the other hand we have
\begin{align*}
\int_X\{u, v\}\,dH^n_F&=\int_X(\mathrm{grad}'u)v\,dH^n_F-\int_X(\mathrm{grad}'v)u\,dH^n_F\\
&=\int_Xh_X(\overline{\partial}u, \overline{\partial}\overline{v})\,dH^n_F-\int_Xh_X(\overline{\partial}v, \overline{\partial}\overline{u})\,dH^n_F\\
&=\int_X(\Delta_{\overline{\partial}}^Fu) v\,dH^n_F-\int_X(\Delta_{\overline{\partial}}^Fv) u\,dH^n_F\\
&=\int_X\,uv\,dH^n_F-\int_X\,uv\,dH^n_F=0.
\end{align*}
Thus Proposition \ref{Poisson} shows $\Delta^F_{\overline{\partial}}\{u, v\}-\{u, v\}=0$, i.e.
$\{u, v\} \in \Lambda_1$.
Thus we have (1).

Finally if (5.3a) holds, then the equivalence between (2) and (3) follows from Proposition \ref{suffsobo}, 
(2).
\end{proof}
Compare the following with Corollary \ref{lielambda}.
\begin{proposition}\label{liealg}
Let $(X, g_X, J, F)$ be an almost smooth Fano-Ricci
limit space.
Moreover we assume that one of the following holds:
\begin{enumerate}
\item[(1)] $\Lambda_1 \subset \mathrm{LIP}_{\bfC}(X)$.
\item[(2)] All $L^1$-holomorphic vector fields on $\mathcal{R}$ are in $L^2_{\bfC}(T_{\bfC}X)$ with (5.3a).
\item[(3)] $F \equiv 0$.
\end{enumerate}
Then $(\Lambda_1, \{\cdot, \cdot \})$ is a Lie algebra.
In particular, if $(X, g_X, J, F)$ is a K\"ahler-Ricci limit soliton, then we have the decomposition for $\mathfrak{h}_1(X)$ as in Theorem \ref{decomp}.
\end{proposition}
\begin{proof}
If (2) holds, then the assertion follows directly from Proposition \ref{suffsobo} and (1) of Remark \ref{pointwise}.

Next we assume that (1) holds.
Note that by Propositions \ref{Laplacianexplicite} and \ref{Laplacianequal}, we have $\Lambda_1 \subset \mathcal{D}^2_{\bfC}(\Delta, X)$.
Let $u, v \in \Lambda_1$.
Then by \cite[Theorem $4.12$]{Honda14b} we have $\mathrm{Hess}_u, \mathrm{Hess}_v \in L^2_{\bfC}(T_{\bfC}^*X \otimes T^*_{\bfC}X)$.
In particular $\nabla \{u, v\} \in L^2(X)$.
Thus by Remark \ref{rem}, we have $\{u, v\} \in H^{1, 2}_{\bfC}(X)$.
Therefore Propositions \ref{isometry} and \ref{equivlie} show that $(\Lambda_1, \{\cdot, \cdot \})$ is a Lie algebra.

Finally if (3) holds, then Theorem \ref{decomp} and \cite[Theorem $7.9$]{CheegerColding3} yield that (1) holds.
This completes the proof.
\end{proof}

\subsection{Remarks on the Lie algebra structure of $\mathfrak{h}_{1}(X)$ on almost smooth setting}
\begin{proposition}\label{lieequivalence}
Let $(X, g_X, J, F)$ be an almost smooth Fano-Ricci
limit space with the assumption (5.3a).
Then for any complex subspace $\Lambda$ of $\Lambda_1$ and $u, v \in \Lambda_1$, the following are equivalent:
\begin{enumerate}
\item $\{u, v\} \in \Lambda$.
\item $[\mathrm{grad}'u, \mathrm{grad}'v] \in \mathfrak{h}^{\Lambda}(X)$.
\end{enumerate}
In particular, $(\Lambda, \{\cdot, \cdot \})$ is a Lie algebra if and only if $(\mathfrak{h}^{\Lambda}(X), [ \cdot, \cdot ])$ is a Lie algebra.

Moreover if $\Lambda = \Lambda_1$, then the conditions above are equivalent to the following:
\begin{enumerate}
\item[(c)] $[\mathrm{grad}'u, \mathrm{grad}'v] \in L^2_{\bfC}(T_{\bfC}X)$.
\end{enumerate}
\end{proposition}
\begin{proof}
Proposition \ref{isometry} yields that if (a) holds, then (b) holds.
Thus we assume that (b) holds.

Then by (1) of Remark \ref{pointwise}, there exists $w \in \Lambda$ such that $\mathrm{grad}'\{u, v\}=\mathrm{grad}'w$.
In particular, $\mathrm{grad}'\{u, v\} \in L^2_{\bfC}(T'X)$.
Since $\{u, v\} \in L^1_{\bfC}(X)$, Proposition \ref{suffsobo} yields $\{u, v\} \in H^{1, 2}_{\bfC}(X)$.
Thus by Propositions \ref{L2norm} and \ref{Poincare}, we see that $\{u, v\}-w$ is constant.
Since
\[\int_X\{u, v\}\,dH^n_F=\int_Xw\,dH^n_F=0,\]
we have $\{u, v\}=w$ which gives (a).

It also follows from the argument above that if $\Lambda = \Lambda_1$ and (c) hold, then (a) holds.
This completes the proof.
\end{proof}
\begin{proposition}
Let $(X, g_X, J, F)$ be an almost smooth Fano-Ricci
limit space with the assumption (5.3a).
Moreover we assume that $\mathfrak{h}_1(X) \subset L^{\infty}_{\bfC}(T'X)$.
Then $(\mathfrak{h}_1(X), [ \cdot, \cdot ])$ is a Lie algebra.
\end{proposition}
\begin{proof}
Let $u, v \in \Lambda_1$.
Proposition \ref{Laplacianexplicite} yields $u, v \in \mathcal{D}^2_{\bfC}(\Delta, X)$.
In particular by \cite[Theorem $4.11$]{Honda14b} we have $\mathrm{Hess}_u, \mathrm{Hess}_v \in L^2_{\bfC}(T^*_{\bfC}X \otimes T^*_{\bfC}X)$.
Since
\[[ \mathrm{grad}'u, \mathrm{grad}'v]=\nabla_{\mathrm{grad}'u}\mathrm{grad}'v-\nabla_{\mathrm{grad}'v}\mathrm{grad}'u,\]
we have $[\mathrm{grad}'u, \mathrm{grad}'v] \in L^2_{\bfC}(T_{\bfC}X)$.
By Proposition \ref{lieequivalence} this completes the proof.
\end{proof}

\section{Decomposition theorem for Ricci limit $\Q$-Fano spaces}
In this section we consider the case when the Fano-Ricci limit space is a $\Q$-Fano variety.
Let $X$ be an $m$-dimensional $\Q$-Fano variety, that is, $X$ is a normal projective variety whose anti-canonical divisor $K_X^{-1}$ is an ample $\Q$-Cartier divisor. Fix a sufficiently large integer $m$, so that $K_X^{-m}$ is a very ample Cartier divisor. Let $\Phi_m\colon X\to \P^N(\C)$ be the Kodaira embedding defined by $K^{-m}_X$. 
Let $\mathfrak{hol}(X)$ be the Lie algebra of all holomorphic vector fields on $X$. By the normality of $X$, $\mathfrak{hol}(X)$ is isomorphic to $\mathfrak{hol}(X_0)$ where $X_0$ is the regular part as 
a $\Q$-Fano variety. Note that $\mathfrak{hol}(X)$ is a Lie subalgebra of $\mathfrak{pgl}(N+1, \C)$. In particular, $\mathfrak{hol}(X)$ is finite-dimensional. 
We say that $X$ has a structure of Ricci limit $\Q$-Fano space if 
$(X, g_X, dH)$ a Fano-Ricci limit space of a sequence of 
$m$-dimensional Fano manifolds $(X_i, g_{X_i}, J_i, F_i)$ and 
the regular part $\mathcal{R}$ as a Ricci limit space coincides with the 
regular part $X_0$ as a $\Q$-Fano variety and the metric $g_X$ on $X_0$ is smooth, and
satisfies (5.1a). 
More precisely, the conditions (2.1a) - (2.1e), (4.1a), (4.1b), $\mathcal{R} = X_0$ and (5.1a) are satisfied.
In order for $(X, g_X, dH)$ to satisfy the condition of an almost smooth Fano-Ricci limit space,
it has to satisfy (5.1b), and (5.1c).
However because of normality $X$ satisfies
\begin{enumerate}
\item[(6.1a)]\label{strongliu} Every holomorphic function on $\mathcal{R}$ is constant,
\end{enumerate}
and thus (5.1b) is satisfied trivially. 
\begin{proposition}\label{qfano}  A Ricci limit $\Q$-Fano 
space is an almost smooth Fano-Ricci 
limit space.
\end{proposition}
\begin{proof}
It suffices only to show (5.1c).
Let $u \in H^{1, 2}_{\bfC}(X)$ with 
$\mathrm{grad}'u \in \mathfrak{h}_{reg}(X)$.
Then by a simple calculation we have
\[\overline{\partial}(\Delta^F_{\overline{\partial}}u-u)=0\]
on $\mathcal{R}$.
Thus by (6.1a), we see that $\Delta^F_{\overline{\partial}}u-u$ is constant.
In particular, $\Delta^F_{\overline{\partial}}u \in L^2(X)$.
Thus Proposition \ref{lapdom} yields $u \in \mathcal{D}^2_{\bfC}(\Delta^F_{\overline{\partial}}, X)$. 
This completes the proof.
\end{proof}
We can use Phong-Song-Sturm's compactness \cite{PhongSongSturm} to show
the following compactness of decomposition theorems for K\"ahler-Ricci solitons.
\begin{theorem}\label{decompconti}
Let $(X_i, g_{X_i}, J_i, F_i)$ be a sequence of K\"ahler-Ricci solitons 
with $\mathrm{Ric}_{X_i} \ge K$, $H^n(X_i) \ge v$, $\mathrm{diam}\,X_i \le d$, and $F_i \ge c$.
Then 
there exist a subsequence $i(j)$, a Ricci limit $\bfQ$-Fano space $(X, g_X, J, F)$ of the subsequence 
$(X_{i(j)}, g_{X_{i(j)}}, J_{i(j)}, F_{i(j)})$, and the limit 1-eigenspace $\Lambda$ of $\Lambda_1(X_{i(j)})$ 
such that 
\[\sup_j||\barpartial F_{i(j)}||_{L^{\infty}}<\infty,\]
that $(X, g_X, J, F)$ is a K\"ahler-Ricci limit soliton, 
that $-\mathrm{grad}'F$ acts on $\mathfrak{h}^{\Lambda}(X)$ by the adjoint action, 
and that the spectral convergence for the adjoint actions of $- \grad'\,F_{i(j)}$ holds, i.e.
\[\lim_{i \to \infty}\lambda_j(X_i, F_i)=\lambda_j(X, F) \le C(n, K, d, v, c),\]
where $\lambda_j(X, F)$ denote the $j$-th eigenvalue of the adjoint action of $- \grad'\,F$ 
counted with multiplicity.
Moreover the decomposition as in Theorem \ref{decomp} holds for $\mathfrak{h}^{\Lambda}(X)$, 
$\mathfrak{h}^{\Lambda}_0(X)$ is reductive, and $\mathcal{F}_X|_{\mathfrak{h}^{\Lambda}(X)}$ is a 
character of $\mathfrak{h}^{\Lambda}(X)$ as a Lie algebra.
\end{theorem}
\begin{proof}
By (\ref{14}), Corollaries \ref{compactness}, \ref{L2est}, Propositions \ref{poissonstr}, \ref{qfano}, Theorem \ref{decomp} and \cite[Theorem $1.2$]{PhongSongSturm}, it suffices to check that $\mathfrak{h}^{\Lambda}_0(X)$ is reductive (note that the assumption on upper bounds for Futaki invariants in \cite[Theorem $1.2$]{PhongSongSturm} is satisfied by (\ref{14})).

Let $V=\mathrm{grad}'u \in \mathfrak{h}^{\Lambda}_0(X)$ for some $u \in \Lambda$.
By the normality of $X$, $V$ is a restriction of a holomorphic vector field $W$ on $\bfP^N(\bfC)$.
This can be seen as follows. First, by normality $V$ extends to the singular set of $X$ and the 
one parameter group acts
on the sections of the pluri-anticanonical bundle. By Kodaira embedding it induces a one parameter group of projective transformations. Let $W$ be its infinitesimal vector field.
The restriction of $W$ to the embedded $X$ coincides with $V$.

In particular $\mathrm{Re}(W)$ is complete on $\bfP^N(\bfC)$.
This implies that $\mathrm{Re}(\mathrm{grad}'u)$ is complete on $\mathcal{R}$.
Therefore Theorem \ref{decomp} with Remark \ref{completeremark} yields the assertion. 
\end{proof}

\begin{remark}For a sequence of compact shrinking solitons with uniformly bounded potentials
$||F_i||_{L^\infty} < C$, the diameters of the sequence are uniformly bounded
by \cite{WeiWylie09, Limoncu, Tadano1504}. Thus in Theorem 
\ref{decompconti} the diameter bound $\mathrm{diam}\,X_i \le d$ follows from $F_i \ge c$ and 
Corollary \ref{compactness}.
Uniform lower bound of dimeters is always satisfied without any assumption on the potential $F_i$
(see \cite{FutakiSano1007, FLL12}).
\end{remark}

We also have the following decomposition theorem.
\begin{theorem}\label{mainth} If a Ricci limit $\Q$-Fano space is a K\"ahler-Ricci limit soliton and if all
holomorphic vector fileds on $X$ are $L^2$ and with smooth poteitials on the regular set, 
i.e. $\mathfrak{hol}(X) = \mathfrak{h}_{reg}(X)$, then
$\mathfrak{hol}(X)$ has the same structure as a smooth K\"ahler-Ricci soliton. That is,
\[\mathfrak{hol}(X)=\mathfrak{hol}_0(X) \oplus \bigoplus_{\alpha >0}\mathfrak{hol}_{\alpha}(X),\]
where $\mathfrak{hol}_{\alpha}(X)$ is the $\alpha$-eigenspace of the adjoint action of $-\mathrm{grad}'F$.
Furthermore,  $\mathfrak{hol}_0(X)$ is a maximal reductive Lie subalgebra.
\end{theorem}
\begin{proof} By Proposition \ref{qfano}, a Ricci limit $\Q$-Fano space is an almost smooth
Ricci limit space. By Proposition \ref{h1h} we have $\mathfrak{h}_{reg}(X) = \mathfrak{h}_1(X)$.
But by our assumption $\mathfrak{h}_{reg}(X) = \mathfrak{hol}(X)$, which is naturally a Lie algebra.
Thus 
taking $\Lambda$ to be $\Lambda_1$, 
(1) and (2) in Theorem \ref{decomp} are satisfied. Then our theorem 
is a direct consequence of Theorem \ref{decomp}.
\end{proof}
The case of smooth K\"ahler-Ricci solitons have been obtained in \cite{TianZhu00}. 
\begin{remark}It is known by Berman and Witt Nystr\"om \cite{BermanWittNystrom14} that
if a $\mathbb{Q}$-Fano variety $X$ admits a K\"ahler-Ricci soliton with the 
K\"ahler potential extended continuously on the whole $X$ then $\mathfrak{hol}_0(X)$ is 
reductive. 
\end{remark}
\begin{remark}
It is not known when all holomorphic vector fields on $X$ are $L^2$ with respect to the Ricci limit
space structure, or when $L^2$ holomorphic vector fields consist a Lie algebra. 
By Remark 3.7 and the normality, the condition (5.3a) is satisfied, and the results in 
subsection 5.4 can be applied.
\end{remark}

\bibliographystyle{amsalpha}

\end{document}